\documentclass{article}

\usepackage[english]{babel}

\usepackage[utf8]{inputenc}
\usepackage{stmaryrd}
\usepackage[utf8]{inputenc}
\usepackage{amsmath}
\usepackage{graphicx}
\usepackage{amssymb}
\usepackage{amsthm}
\usepackage{tikz-cd}
\usepackage{mathrsfs}
\usepackage[colorinlistoftodos]{todonotes}
\usepackage{enumitem}
\usepackage{yfonts}
\usepackage{ dsfont }
\usepackage{MnSymbol}
\usepackage{slashed}

\date{\today}
\newtheorem{thm}{Theorem}[section]
\newtheorem{lem}[thm]{Lemma}

\theoremstyle{definition}
\newtheorem{eg}[thm]{Example}

\newtheorem{cor}[thm]{Corollary}

\newtheorem{rmk}{Remark}
\newtheorem{prop}[thm]{Proposition}
\newtheorem{Def}[thm]{Definition}
\newtheorem{Question}{Question}

\newtheorem*{Acknowledgement}{Acknowledgement}

\newcommand{\cf}{\emph{cf.} }

\newcommand{\R}{\mathbb{R}}
\newcommand{\C}{\mathbb{C}}
\newcommand{\Z}{\mathbb{Z}}

\newcommand{\Q}{\mathbb{Q}}

\newcommand{\norm}[1]{\left\lVert#1\right\rVert}

	\title{Calibrated submanifolds, adiabatic limit, and gradient graphs}
	\author{Yang Li}

\begin{document}
		
		\maketitle

		\begin{abstract}
		We study the compactness question for certain special Lagrangians in semiflat SYZ fibrations (resp. associative submanifolds in Donaldson's proposal of collapsing coassociative K3 fibrations), and give some criterion for when gradient graphs emerge from the adiabatic limit. This gives a partial converse to the Donaldson-Scaduto proposal.
		\end{abstract}

		\section{Introduction and background}

		In many problems of calibrated geometry, one observes the emergence of \emph{gradient graphs} in the adiabatic limit, where the ambient manifold collapses to a lower dimensional manifold. We are primarily interested in three notable instances of this phenomenon:
		
		\begin{itemize}
			\item The emergence of tropical curves as adiabatic limits of holomorphic curves;

			\item Certain special Lagrangian submanifolds inside an SYZ fibration, whose projection to the SYZ base is a small thickening of a one-dimensional graph;
			
			\item The Donaldson-Scaduto picture of associative submanifolds inside $G_2$-manifolds.

		\end{itemize}

		The first case is already well understood via algebro-geometric and symplectic techniques, and the last two cases occur in some gluing constructions \cite{ChiuLiLin,ChiuLin}, in which balanced gradient graphs arise as part of the gluing data. The goal of this paper is an attempt to answer the converse compactness question: 
		
		\begin{Question}\label{Question}
			What conditions on the calibrated submanifolds can a priori ensure that balanced gradient graphs must appear in the adiabatic limit?
		\end{Question}

	%	We will now review the three cases from the constructive viewpoint.

		\subsection{Holomorphic curves and tropical curves}

		The simplest classical case concerns holomorphic curves in $(\C^*)^2$. We have a one-parameter family of rescaled logarithm maps
		\[
		\text{Log}_t: (\C^*)^2\to \R^2, \quad (z_1, z_2)\mapsto \frac{1}{|\log t|}(\log |z_1|, \log |z_2|).
		\]
		Now consider the family of holomorphic curves
		\[
		f_t(z)=\sum_{m\in \Delta} a_m t^{-\nu(m)} z^m =0,
		\]
		where $\Delta$ is a finite integral convex polytope in $\Z^2$ (called the Newton polytope), and $a_m\neq 0$, $\nu(m)\in \R$. The image of these holomorphic curves under $\text{Log}_t$ is known as the amoeba, and as $t\to 0$ it converges in the Hausdorff topology to the \emph{tropical curve}, or equivalently the non-smooth locus of the tropical polynomial
		\[
		\text{Trop}(f)= \max_{m\in \Delta} \langle m, x\rangle + \nu(m).
		\]
		This consists of a finite number of (possibly infinite) straight edges, given by 
		\[
		\langle m_1, x\rangle+ \nu(m_1)= \langle m_2, x\rangle+ \nu(m_2)= \max_{m\in \Delta}  \langle m, x\rangle + \nu(m),
		\]
		joined at a finite number of vertices.

		\begin{rmk}
			Tropical degenerations of holomorphic curves has been developed into a highly elaborate theory. Some notable applications include the tropical correspondence theorem for Gromov-Witten invariants by Mikhalkin \cite{Mikhalkin}, the theory of exploded manifolds by B. Parker \cite{Parker}, and the Gross-Siebert program on mirror symmetry \cite{GrossSiebert}, to name just a few. There are multiple approaches including tropical geometry, J-holomorphic curve theory, log geometry, and non-archimedean geometry. We will not attempt to improve on the existing theory, but the reader should keep tropical curves in mind as the simplest manifestation of the emergence of graphs in the adiabatic limit.
		\end{rmk}

		\subsection{Special Lagrangians inside semi-flat SYZ fibrations}

		Let $(X, J,\omega, \Omega)$ be an $n$-dimensional \emph{almost Calabi-Yau} manifold, namely $(X, J, \omega)$ is K\"ahler, and $\Omega$ is a nowhere vanishing holomorphic volume form, such that $\frac{1}{n!} \omega^n=  \frac{\sqrt{-1}^{n^2}} {2^n}e^{n\rho} \Omega \wedge \overline{\Omega}$. Recall that a \emph{special Lagrangian} submanifold $L$ of phase $\hat{\theta}$ inside a Calabi-Yau $n$-fold $(X, I, \omega, \Omega)$ is a real $n$-dimensional submanifold, such that
		\[
		\omega|_L=0,\quad \text{Im}(e^{-i\hat{\theta}}\Omega)|_L=0.
		\]
		Up to a choice of orientation, $L$ is calibrated by the closed $n$-form $\text{Re}(  e^{-i\hat{\theta}}\Omega ) $ with respect to the Riemannian metric $\tilde{g}_X= e^{-\rho} g_X$.

		We consider the following \emph{semi-flat} metric ansatz. Let $B$ be a contractible open domain in $\R^n$,  carrying a Hessian metric
		\begin{equation}
		g= g_{ij} dx_i dx_j= \frac{\partial^2 u}{\partial x_i\partial x_j} dx_idx_j,
		\end{equation}
		where $u$ is a smooth and convex potential function. (In the Calabi-Yau case, $u$ solves the real Monge-Amp\`ere equation $\det(D^2 u)=1$ in the affine coordinates on $B$, but we prefer to work in the more general setting.)  
		We take $X= T^n\times B$  equipped with the K\"ahler structure
		\begin{equation}
		\begin{cases}
				\omega_\epsilon= \epsilon \sum g_{ij} dx_i \wedge dy_j,
				\\
				\Omega_\epsilon= \bigwedge_1^n (\epsilon dy_i- \sqrt{-1}  dx_i),
				\\
				 g_\epsilon=\sum g_{ij}(dx_idx_j+\epsilon^2 dy_i dy_j),
				 \\
				 e^{n\rho}= \det(g)= \det(D^2 u).
		\end{cases}
		\end{equation}
		%\[
		%\omega= \sum dx_i \wedge dy_i, \quad \Omega= \bigwedge_1^n (dy_i+ \sqrt{-1} g_{ij} dx_j),\quad g= g_{ij}dx_idx_j+ g^{ij}dy_i dy_j,
		%\]
		where %$y_i dx_i$ is the natural Liouville 1-form on $T^* B$, and 
		$y_i$ are $\Z^n$-periodic fibre coordinates on $T^n$. In particular, the fibres of $\pi: X\to B$ are special Lagrangians of phase zero, and $X$ collapses down to $(B, g_{ij}dx_idx_j) $ as the fibre length parameter $\epsilon\to 0$.

		\begin{rmk}
			When $B=\R^n$, and $u= \frac{1}{2} |x|^2$, 
			then $X$ is $T^n\times \R^n$ 
			equipped with a Euclidean metric, where $T^n$ has length scale of order $\epsilon$. More generally, the SYZ conjecture philosophy \cite{SYZ} predicts that semi-flat Calabi-Yau metrics are up to small error the local models for the generic regions of compact polarised Calabi-Yau manifolds near the large complex structure limit  \cite{Livaluative}. 
		\end{rmk}

		We now seek special Lagrangians $L$ of phase $\hat{\theta}=-\frac{\pi}{2}$, whose projections to $B$ is a small thickening of some one-dimensional graph. We now give an informal discussion about the gluing construction in  \cite{ChiuLiLin}. %We assume $\hat{\theta}\neq 0$ mod $\pi \Z$. 

		Near the edge $e$ of the graph, but staying away from the vertices, $L$ is modelled on 
		 a $T^{n-1}$-bundle over the edge $e$, where the $T^{n-1}$-fibres are flat subtori inside the $T^n$ fibres of $\pi: X\to B$. Taking a given $T^{n-1}$-fibre as the initial data, we can locally view $L$ as the evolution of $T^{n-1}$. Let $f_1,\ldots f_{n-1}$ be the generators of the $T^{n-1}$-subtorus, and $\delta(e)=\bigwedge f_i\in H_{n-1}(T^n) \simeq H^1(T^n)$ specifies its homology class, which induces an element of $H^1(T^n,\R)\simeq T_x^* B$ still denoted $\delta(e)$. The $(n-1)$-conditions
		\[
		\omega(f_i,\cdot)=0,\quad i=1,\ldots n-1
		,\]
		specifies a unique tangent direction of $T_xB$ up to scale, equivalently described as the dual element of $\delta(e)$ with respect to the metric $g_{ij}dx_idx_j$ on $B$. The one additional condition
		\[
		\text{Im}(e^{-i\hat{\theta}} \Omega)|_L=0
		\]
		amounts to the constancy of the angular coordinate on $T^n/T^{n-1}$. %(This requires the assumption that $\hat{\theta}\neq 0$ mod $\pi \Z$, because otherwise the solution of the initial value problem would simply be the SYZ fibre $T^n$, which is a special Lagrangian with phase zero.) 
		The upshot is that the edge $e$ follows the flow trajectory
		\begin{equation}\label{gradientflow1}
			\frac{dx_i}{dt}=\sum_j g^{ij} (x) \delta(e)_j,
		\end{equation}
		where $g^{ij}$ is the inverse matrix of $g_{ij}$. In other words, the edge follows the \emph{gradient flowline} of the linear function $x\mapsto \langle \delta(e), x\rangle$.

		Near the vertices several edges $e$ can come together, and we need special Lagrangians with several cylindrical ends as a gluing model; it turns out that the hyperk\"ahler rotation of algbraic curves in $(\C^*)^2$ provides a large supply of such local models \cite{ChiuLiLin}. The homology classes of the several $T^{n-1}$ ends associated to the outgoing egdes $e$, must sum to zero, giving rise to a \emph{balancing condition} at each vertex $v$,
		\begin{equation}\label{balancing1}
			\sum_{ \text{$e$ adjacent to $v$} } \delta(e)=0.
		\end{equation}
		A graph satisfying the gradient flow equation (\ref{gradientflow1}) along edges, and the balancing condition (\ref{balancing1}) at the vertices, shall be called a \emph{balanced gradient graph}.

		\subsection{Donaldson-Scaduto picture}

		Donaldson \cite{Donaldson} proposed a program to study $G_2$-manifolds with adiabatic coassociative K3 fibrations. We will work in a slightly more general setting, requiring the closedness of the $G_2$-structure, but not the torsion-free condition.

		Let $\pi: X\to B$ be a smooth  K3 fibration over a contractible open domain $B\subset \R^3$, with local coordinates $x_1,x_2,x_3$. Suppose $H: B\to H^2(K3,\R)$ is a positive section, namely the pullback metric of the intersection form
		\[
	g=	\frac{\partial H}{\partial x_i}\cdot 	\frac{\partial H}{\partial x_j} dx_i\otimes dx_j
		\]
	is positive definite on $B$. For each $p\in B$, after a linear change of coordinates, we may assume $\frac{\partial}{\partial x_i}$ is a $g$-orthonormal basis of $T_p B$. The cohomology classes 
	\[
[	\omega_i]:=  \frac{\partial H}{\partial x_i} \in H^2(K3,\R),\quad i=1,2,3,
	\]
		span a positive definite subspace in $H^2(K3,\R)$. We assume that for any $p\in B$,  this 3-dimensional subspace is not orthogonal to any  $(-2)$-classes in $H^2(K3,\Z)$.
		So by the Torelli theorem, this determines up to fibre diffeomorphism a hyperk\"ahler triple $(\omega_1, \omega_2, \omega_3) $ on the K3 surface fibre, where $\omega_i$ is the hyperk\"ahler symplectic form in the class $[\omega_i]$.

	Upon choosing a horizontal distribution $\mathcal{H}$, we can regard $\omega_i$ as a vertical 2-form on the total space $X$, and consider $ \sum_1^3 \omega_i\wedge  dx_i$, which is independent of the linear change of coordinates. For an appropriate choice of $\mathcal{H}$, we have $d  (\sum_1^3 \omega_i\wedge  dx_i)=0$  modulo a horizontal-vertical type $(3,1)$-form. Upon adding $f dx_1\wedge dx_2\wedge dx_3$ for a suitable function $f$, we can modify $ \sum_1^3 \omega_i\wedge  dx_i$ into a closed 3-form which we denote by $\underline{\omega}$.
		We write $\underline{\lambda}= - \sqrt{\det(g)} dx_1dx_2dx_3$, which pulls back to a 3-form on $X$. Then 
	\[
	\phi_\epsilon := \epsilon^2 \underline{\omega}+ \underline{\lambda}
	\]
	defines  a closed $G_2$-structure depending on the small parameter $\epsilon>0$, and in particular determines a Riemannian metric $g_\epsilon$ on $X$ (as a caveat, this 3-form is in general not co-closed).

	Along a given fibre $\pi^{-1}(p)$, upon taking $g$-orthonormal base coordinates $x_1, x_2, x_3$ at $p\in B$, the Riemannian metric is $g_\epsilon= \epsilon^2 g_{K3}+ \sum_1^3 dx_i^2$ (modulo $O(\epsilon^2)$ error coming from the correction $fdx_1\wedge dx_2\wedge dx_3$ above), where $g_{K3} $ is the hyperk\"ahler metric on the given K3 fibre determined by $(\omega_1, \omega_2,\omega_3)$ with total volume 
	\[
\text{Vol}_{g_{K3}}(\pi^{-1}(p) )=	\int_{K3} \frac{\omega_i^2}{2}= \frac{1}{2},\quad i=1,2,3.
	\]
	By construction, the 3-form $\phi_\epsilon$ vanishes on the K3 fibres, so the K3 fibres are coassociative submanifolds.

	The 3-dimensional submanifolds $L\subset M$ are called \emph{associative} if they are calibrated by the closed 3-form $\phi_\epsilon$. Donaldson-Scaduto \cite{DonaldsonScaduto} suggested a way to produce gluing ansatz for associative submanifolds: one starts with a graph $\Gamma$ embedded in the base $B$, whose edges $e$ are the \emph{gradient flowlines} of the functions $H\cdot \delta(e)$ for some locally constant homology classes $\delta(e)\in H_2(K3,\Z)$, and the vertices $v$ of the graph satisfy the balancing condition
	\[
	\sum_{e \text{ adjacent to } v} \delta(e)=0\in H_2(K3, \Z).
	\]
	We call this data a \emph{balanced gradient graph}.

	The associative submanifold $L$ is approximately a fibration over the edges of $\Gamma$, whose fibres are holomorphic curves inside the hyperk\"ahler K3 surfaces in the homology class $\delta(e)$, while in the neighbourhood of the vertices, $L$ is locally modelled on an asymptotically cylindrical associative submanifold inside $K3\times \R^3$.

		\begin{rmk}
			In Donaldson's programme on adiabatic limits of $G_2$-manifold, the torsion free condition on the $G_2$-structure leads to the requirement that the positive section $H$ is a \emph{maximal submanifold}, namely its image  has zero mean curvature inside the flat vector space $H^2(K3, \R)$ with indefinite signature $(3, 19)$. By an appropriate choice of the horizontal connection $\mathcal{H}$, the closed $G_2$-structure would be approximately coclosed, and then one can hope to perturb $\phi_\epsilon$ to a torsion free $G_2$-structure.

			The more sophisticated part of Donaldson-Scaduto \cite{DonaldsonScaduto} speculates on the question of compactifying the associative submanifold inside the $G_2$-manifolds, and what can happen in a generic one-parameter family of associative submanifolds. % This involves the following phenomena:
			%\begin{enumerate}
			%    \item  To compactify the ambient $G_2$-manifolds, one needs to allow for singular fibres. The singularity of the fibration is a union of circles, which are fibred over a link inside the 3-dimensional base $B$, and the singularity is transversely modelled on the Lefshetz fibration $\C^3\xrightarrow{\sum z_i^2} \C$. 

			%    \item  The gradient flowlines can end on the link in $B$, and the associative submanifold is locally modelled on the Lefschetz thimble.
			
			%\item

			%\end{enumerate}
		\end{rmk}

		\begin{rmk}
			The Donaldson-Scaduto story has an analogue in the world of special Lagrangians. If we assume the $G_2$-structure is the product of a Calabi-Yau 3-fold with $S^1$, then the problem reduces to finding special Lagrangian submanifolds inside a Calabi-Yau 3-fold with a holomorphic K3 fibration over $\mathbb{CP}^1$, where the K3 fibres are shrinking down as $\epsilon\to 0$;  a special case is treated in the gluing construction of Chiu-Lin \cite{ChiuLin}. In another variant situation, the K3 fibres are replaced by ALE or ALF spaces, so that the ambient Calabi-Yau 3-fold is an ALE/ALF gravitational instanton fibration over a Riemann surface. The stability conditions on the Fukaya category, which are conjecturally related to the special Lagrangians in these Calabi-Yau manifolds, have been considered in \cite{Smith, BridgelandSmith}.
		\end{rmk}

		\subsection{Main results}

	Henceforth we consider a sequence of  special Lagrangian cycles $L_i$ of phase $\hat{\theta}=- \frac{\pi}{2}$  in the semiflat SYZ fibration on $X$ (resp. associative cycles $L_i$ inside Donaldson's collapsing K3 fibrations), where the ambient parameter $\epsilon\to 0$ in the $i\to +\infty$ limit. We will sometimes suppress the sequential index for brevity. Recall that $\pi: X\to B$ is the projection map to the base.

	\begin{Def}\label{Def:uniformvolumebound}
	We say the special Lagrangians $L_i$ inside $(X, g_\epsilon, \omega_\epsilon, \Omega_\epsilon)$ (resp. associative $L_i$ inside $(X, g_\epsilon, \phi_\epsilon)$) satisfy the \emph{uniform linear	volume bound}, if $\epsilon \to 0 $ as $i\to +\infty$, and 
	for each given compact subset $K\subset B$, there is 
	some uniform constant $r_0, C$ depending on $K$ but independent of $i,\epsilon$, such that $B_g(p, r_0)$ is properly contained in $B$, and 
	\begin{equation}
\sup_{   p\in K }	\epsilon^{1-n} \norm{L_i} (  \pi^{-1} B_g(p, r_0)   ) \leq C.
	\end{equation}
	resp. 
	\[
	\sup_{p\in K }	\epsilon^{-2} \norm{L_i}( \pi^{-1} B_g(p, r_0)   ) \leq C.
	\]
	\end{Def}

		\begin{rmk}
	The intuition is that the special Lagrangian (resp. associative) is supposed to have volume growth only along one base dimension, while the remaining $(n-1)$ (resp. two) fibre directions have length scale $O(\epsilon)$, and contributes the factor $\epsilon^{n-1}$ in the special Lagrangian case (resp. $\epsilon^2$ in the associative case).

		\end{rmk}

		\begin{rmk}
			(Compactification)
		Suppose the semiflat SYZ fibration $X\to B$ (resp. Donaldson's adiabatic K3 fibration) 
		arise as an open subset of a global fibration over a compact base $\bar{X}\to \bar{B}$. 
		 and the special Lagrangians (resp. associative submanifolds) on $X$ are obtained by restrictions of global examples $\bar{L}\subset \bar{X}$. Then to check the uniform linear mass bound, one only needs to control the global integral $\int_{\bar{L}} \text{Re}(e^{-i\hat{\theta}} \Omega)$ (resp.  $ \int_{\bar{L} }\phi_\epsilon$), which only depends on the homology class of $\bar{L}$. 
		\end{rmk}

		We are interested in capturing the \emph{large scale limiting behaviour} of the sequence $L_i$. We recall the notion of  rectifiable 1-currents valued in a finite rank lattice $\Lambda$. In our applications, 
		\[
		\Lambda= 
		\begin{cases}
			H_{n-1}(T^n,\Z)\quad &\text{special Lagrangian case},
			\\
			 H_2(K3,\Z),\quad & \text{ associative case}.
		\end{cases}
		\]

		\begin{Def}\label{Def:Lambdavaluedcurrent}
		A \emph{$\Lambda$-valued 1-rectifiable current} $T$ in $B$ is a current of integration over an 1-rectifiable set with a $\Lambda$-valued weight: 
		\begin{equation}
		T(\beta)= \int_\Gamma  \langle \beta(x), \xi(x)\rangle  \delta(x) d\mathcal{H}^1(x),\quad \forall \beta \in \Omega^1_{cpt}(B),
		\end{equation}
		where $\Gamma= \text{spt}(T)\subset B$  is an 1-rectifiable subset,
		$\xi(x)$ is a unit vector spanning the approximate tangent space $T_x \Gamma$ 
	 for $\mathcal{H}^1$-almost every $x\in \Gamma$, and $\delta: \Gamma\to \Lambda$ is a locally integrable function with target in the  lattice $\Lambda$. We say $T$ is closed if 
	 \[
	\partial T(f)= T(df)= 0\in \Lambda\otimes_\Z \R, \quad \forall f\in C^\infty_{cpt}(B).
	 \]

		\end{Def}

		For every $p\in B$ and $\delta\in \Lambda= H_{n-1}(T^n)$ (resp. $\Lambda= H_2(K3,\Z))$, we define 
		\begin{equation}
		|\delta|_p:=  \text{min mass of integral currents in the homology class $\delta$ on  $\pi^{-1}(p)$} ,
		\end{equation}
		where in the special Lagrangian case, the fibre metric is taken to be $g_{ij} dy_i dy_j$ on the torus $\pi^{-1}(p)$, while in the associative case, the fibre metric is the hyperk\"ahler metric $g_{K3}$ on the K3 fibre $\pi^{-1}(p)$  normalised to volume $\frac{1}{2}$.

\begin{eg}\label{eg:minareatorus}
In the special Lagrangian case, given a nonzero $\delta\in H_{n-1}(T^n)$, there is an oriented $\Z$-basis $f_1,\ldots f_n$ for $H_1(T^n)$ such that $\delta= m \wedge_1^{n-1} f_i$ for some $m\in \Z_+$. Then $|\delta|_p$ is achieved by the sum of $m$ parallel copies (counting multiplicity) of the subtorus $T^{n-1}$ generated by $f_1,\ldots f_{n-1}$, and 
\[
|\delta|_p= m \sqrt{  \det( ( g(f_i, f_j) )_{(n-1)\times (n-1)}   ) }=     \sqrt{ \det(g)  g^{ij} \delta_i \delta_j }.
\]
Here we identify $\delta \in H_{n-1}(T^n)$ as an element of $T_p^* B$ via contraction with $\text{Re}(\sqrt{-1} \bigwedge_1^n(dy_i- \sqrt{-1} dx_i) )$. Via the metric $g$, this $\delta$ induces a tangent vector $g^{ij}\delta_j \frac{\partial}{\partial x_i}   
	$ in $T_p B$. For later convenience, we define the unit tangent vector 
\begin{equation}
v_p(\delta):= \frac{ g^{ij}(p) \delta(p)_j \frac{\partial}{\partial x_i}  }{    \sqrt{   g^{ij}\delta(p)_i \delta(p)_j }   }.
\end{equation}

\end{eg}

\begin{eg}\label{eg:K3holocurve}
In the associative case, for a $g$-unit tangent vector $v\in T_p B$, then $\omega_v=\iota_v \underline{\omega}$ gives a well defined K\"ahler form on the hyperk\"ahler K3 fibre $\pi^{-1}(p)$ in the homology class $[\omega_v]= \nabla_v H\in H^2(K3,\R)$, with corresponding complex structure $I_v$. We say $\delta\in H_2(K3,\Z )$ is \emph{admissible}, iff $\delta$
is represented by an $I_v$-holomorphic curve (possibly with multiplicity) for some choice of $v$, iff
\[
|\delta|_p =  \int_\delta \omega_v =    \delta\cdot (\nabla_v H)(p).
\]
In this case, the vector $v$ can be recovered from $\delta$ by the formula
\begin{equation}
v_p(\delta):=   \frac{ \delta\cdot (\nabla H)(p)  }{ |\delta|_p   } ,
\end{equation}
where the dot product is the natural pairing between $\delta \in H_2(K3,\Z)$ and $\nabla H\in H^2(K3,\R)\otimes T_p B$. .
\end{eg}

\begin{Def}\label{Def:gradientcycle}
Let $T$ be a $\Lambda$-valued 1-rectifiable current. %We define the \emph{density function} $\Theta:\Gamma\to \R_{\geq 0}$ by
%\begin{equation}
%\Theta(x)=  |\delta(x) |_x.
%\end{equation}
%(As a caveat, the density may not be integer valued!) 
The \emph{density measure} for $T$ is the Radon measure 
\[
\mu_T (f)= \int_\Gamma f(x) |\delta(x)|_x d\mathcal{H}^1 (x),\quad \forall f\in C^0_{cpt}(B).
\]

\end{Def}

		\begin{Def}\label{Def:densitymeasure}
	Let $T$ be a  $\Lambda$-valued 1-rectifiable current. We say $T$ is a \emph{$\Lambda$-weighted gradient cycle}, if 
	\begin{enumerate}
		\item  (closedness)  The $\Lambda$-valued current boundary $\partial T=0$.
		\item (gradient flow)  In the special Lagrangian case, we require for $\mathcal{H}^1$-a.e $p\in \Gamma$, the unit tangent vector $\xi(p)  =v_p(\delta)$ (\cf Example \ref{eg:minareatorus}).

		 In the associative case, we require that for $\mathcal{H}^1$-a.e $p\in \Gamma$, the class $\delta(p)\in H_2(K3,\Z)$ is admissible, and $\xi(p)= v_p(\delta)$
			(\cf Example \ref{eg:K3holocurve}).

	\end{enumerate}

		\end{Def}

		\begin{rmk}
		The balanced gradient graphs are special cases of $\Lambda$-weighted gradient cycles. Conversely, if a  $\Lambda$-weighted gradient cycle is supported on an embedded graph in $B$, then it coincides with a balanced gradient graph.
		Here the closedness condition captures the balancing condition at the vertices. 
		\end{rmk}

		Our main result is

		\begin{thm}\label{mainthm}
		Suppose the special Lagrangian integral currents $L_i$ (resp. associatives $L_i$) satisfy the uniform linear mass bound as $\epsilon\to 0$. Then up to passing to subsequence, there exists a $\Lambda$-weighted gradient cycle $T$, such that the followings hold:
		\begin{enumerate}
			\item  The sequence of pushforward measures $\epsilon^{1-n}\pi_* \norm{L_i} $ (resp. $\epsilon^{-2}   \pi_* \norm{L_i} $)    on $B$ converge weakly to the density measure $\mu_T$.

			\item  Suppose $\alpha$ is any compactly supported smooth $n$-form (resp. 3-form), locally of the form $\sum  \alpha_i \wedge dx_i$ where $\alpha_i$ restricts to a closed form on the fibres, so that $\alpha$ naturally defines a compactly supported 1-form $[\alpha]$ on $B$ valued in $H^{n-1}(T^n,\R)$ (resp. $H^2(K3,\R))$, which can be paired with the $\Lambda$-valued current $T$. Then $\int_{L_i}\alpha$ converges to $\langle T, [\alpha] \rangle $.

			\item On any open subset $U$ properly contained in $B$, the projection images $\pi(\text{spt}(L_i)) $ converges in Hausdorff distance  to $\Gamma= \text{spt}(\mu_T)$.
			
		\end{enumerate}
		\end{thm}

	\begin{rmk}
		Our method will be robust against sufficiently small $C^0$ perturbations of the ambient almost Calabi-Yau structure (resp. closed $G_2$-structure). For instance, if we replace the closed $G_2$-structure $\phi_\epsilon$ by $\phi_\epsilon'$ such that $\norm{ \phi_\epsilon-\phi_\epsilon'}_{C^0}\leq \frac{C}{|\log \epsilon|}$ in the $\epsilon\to 0$ limit, and we consider associative integral currents $L_i$ with respect to $\phi_\epsilon'$ instead of $\phi_\epsilon$, then the conclusions of the main theorem still holds (See Remark \ref{rmk:logconvergencerate} below).
	\end{rmk}

		We also discuss the regularity question for the $\Lambda$-weighted gradient cycles. 
		This has some moral similarity with the Allard-Almgren structure theorem \cite{AllardAlmgren} on stationary 1-dimensional varifolds. As a consequence of the monotonicity formula for $\Lambda$-weighted gradient cycles, the density
		\[
		\Theta(p):= \lim_{r\to 0} \frac{1}{2r} \mu_T( B_g(p,r) )
		\]
		is well defined at every point $p\in \Gamma= \text{spt}(T)$. 
	The caveat is that the density $\Theta(p)$ may not be integers, and can vary continuously with $p\in \Gamma$, even though $\delta(p)$ takes value in the discrete lattice $\Lambda$.

		\begin{thm}\label{thm:regularitySlag}
			Suppose $T$ is a $\Lambda$-weighted gradient cycle with support $\Gamma\subset B$.
		In the special Lagrangian case, $\Gamma$ is a \emph{locally finite embedded graph};  in particular, on any compact subset of $B$, it contains only finitely many vertices, and each edge is a gradient flowline.
		\end{thm}

		A similar result holds for the associative case, subject to a technical assumption involving $H_2(K3)$ classes with bounded norm.

		\begin{thm}\label{thm:regularityassociative}
			Suppose $T$ is a $\Lambda$-weighted gradient cycle with support $\Gamma\subset B$.
		In the associative case, suppose the following holds at some $p\in \Gamma$: whenever there is some choice of complex structure $I$ on the hyperk\"ahler K3 surface $\pi^{-1}(p)$, such that two classes $\delta, \delta'\in H_2(K3,\Z)$ both admit an $I$-holomorphic curve representative, with areas $|\delta|_p, |\delta'|_p \leq \Theta(p)$, then $\delta, \delta'$ are proportional.

		Then there is some neighbourhood $U$ of $p\in B$, where $\Gamma\cap U$ is a finite embedded graph.

		\end{thm}

		\textbf{Organization}.  The paper will emphasize the close analogy between the special Lagrangians and the associative cycles $L$. We prove a monotonicity formula for the mass of $L$ over macroscopic balls in $B$ in section \ref{sect:monotonicity}.  This enables us to prove the main theorem \ref{mainthm} in section \ref{sect:limitgradientcycle}. We study the regularity question for $\Lambda$-weighted gradient cycles in section \ref{sect:regularity}, and end with some open discussions.

	%	comparison to Aaron Naber

				\section{Monotonicity formula}\label{sect:monotonicity}

				Since our main theorem is local on the base $B$, it is enough to consider points $p\in B$ 
				lying in a compact subset, and we suppose that $B_g(p,r_0)$ is properly contained in $B$, where $r_0$ is small enough to ensure the local ambient geometry on $\pi^{-1}(B_g(p,r))$ is sufficiently close to the product case.

				We write the base distance function $d_p= \text{dist}_g(p,\cdot)$, which pulls back to a function on $X$. The projection map $\pi: X\to B$ is up to $O(\epsilon)$-error a Riemannian submersion. Up to dilating the base metric by $1+ O(\epsilon)$, we may assume $\pi$ is 1-Lipschitz, so $|\nabla d_p|\leq 1$.

				We first consider the associative case.

				\begin{prop}\label{prop:monotonicityassociative}
				(Monotonicity formula) Suppose $ L$ is an associative cycle inside the closed $G_2$-manifold $(X,\phi_\epsilon)$, then for any $\epsilon\leq r_1<r_2\leq r_0$, there is some uniform constant $C$ such that
				\[
				\begin{split}
					&	e^{Cr_2}	r_2^{-1} 	\norm{L}(  \pi^{-1}B_g(p, r_2)  )- e^{C r_1} r_1^{-1} 	\norm{L}(  \pi^{-1}B_g(p, r_1)  )
					\\
						\geq &  
					\int_{    \text{spt}(L) \cap \pi^{-1} (B(p,r_2)\setminus B(p,r_1) ) }  d_p^{-1} ( 1- |\nabla d_p| ) d\norm{ L}  \geq 0.
				\end{split}
				\]
				\end{prop}

				\begin{proof}
				Without loss $p$ is the origin. We first consider the special case where the $G_2$-structure $\phi_\epsilon$ is the product structure on the central K3 fibre times $\R^3$. Let $x_1,x_2,x_3$ be orthonormal coordinates on $\R^3$, so 
				\[
				\phi_\epsilon= \epsilon^2\sum_i \omega_i \wedge dx_i - dx_1dx_2dx_3=  d\beta, \quad \beta=\sum_{cyc}  x_1( \epsilon^2\omega_1- \frac{1}{3} dx_2\wedge dx_3 ). 
				\]
				Now $d_p= (\sum_1^3 x_i^2)^{1/2}$ is the base distance function pulled back to $X$. 	For a.e. $r<r_0$, we can take the slice 
			$
				\langle L, d_p, r\rangle=  \partial (L\lfloor  \pi^{-1}B_g(p, r))
			$
				and Stokes formula gives	\[
				\norm{L}(  \pi^{-1}B_g(p, r)  )= \int_{L\lfloor \pi^{-1}B_g(p, r) } \phi_\epsilon= \langle L, d_p, r\rangle (   \beta     ),
				\]
				while the coarea formula gives
				\[
				\norm{L}(  \pi^{-1}B_g(p, r)  )=  \int_0^r dt \int_{  \text{spt}(L) \cap \pi^{-1}(\partial B(p,t))  } \frac{1}{|\nabla d_p| } d\norm{ \langle L, d_p, t \rangle}.
				\]

				On the other hand,  we claim that restricted to  $\text{spt}(T) \cap \pi^{-1}(\partial B(p,r))$,
				\begin{equation}\label{eqn:primitive1}
				  |\beta |  \leq r \mathcal{H}^2_{g_\epsilon}.
				\end{equation}
				To see this, by rotational symmetry, we may assume $(x_1, x_2, x_3)=(r,0,0)$. Now 
			$
			\epsilon^2	\omega_1- \frac{1}{3} dx_2\wedge dx_3 
			$
			is bounded by the 2-dimensional area element with respect to the metric $\epsilon^2 g_{K3}+ \frac{1}{3}  \sum dx_i^2 $, which is in turn bounded by our product metric $g_\epsilon= \epsilon^2 g_{K3}+ \sum dx_i^2$.

			Now let $F(r)= \norm{L}(  \pi^{-1}B_g(p, r)  )$. For a.e. $r$, the above implies
			\[
			\begin{split}
				\frac{dF}{dr}= & \int_{  \text{spt}(L) \cap \pi^{-1}(\partial B(r))  } \frac{1}{|\nabla d_p| } d\norm{ \langle L, d_p, r\rangle}
				\\
			\geq &  \int_{  \text{spt}(L) \cap \pi^{-1}(\partial B(r))  } ( \frac{1}{|\nabla d_p| } -1) d\norm{ \langle L, d_p , r\rangle} + r^{-1}  \langle L, d_p, r\rangle (  \beta    )
			\\
			=  &  \int_{  \text{spt}(L) \cap \pi^{-1}(\partial B(r))  } ( \frac{1}{|\nabla d_p| } -1) d\norm{ \langle L, d_p, r\rangle} + r^{-1}  F(r).
			\end{split}
			\]
				Notice that $|\nabla d_p|\leq 1$. Thus
				\[
					\frac{d}{ dr }(\frac{F(r) }{r}  ) \geq r^{-1} \int_{  \text{spt}(L) \cap \pi^{-1}(\partial B(r))  } ( \frac{1}{|\nabla d_p|  } -1) d\norm{ \langle L, d_p, r\rangle}  \geq 0,
				\]
				so upon integration
				\begin{equation}
			\begin{split}
					\frac{F(r) }{r}  \vert_{r_1}^{r_2} 	 & \geq  \int_{r_1}^{r_2} \frac{dt}{t}  \int_{  \text{spt}(L) \cap \pi^{-1}(\partial B(t))  } ( \frac{1}{|\nabla d_p | } -1) d\norm{ \langle L, d_p, t \rangle} 
			\\
			& =\int_{    \text{spt}(L) \cap \pi^{-1} (B(r_2)\setminus B(r_1))  }  d_p^{-1} ( 1- |\nabla d_p| ) d\norm{ L} 
		\geq 0.
			\end{split}
				\end{equation}

				The more general case of non-product $G_2$-structure deviates from a product structure by relative metric error of order $O(r)$ on the preimage of $\partial B(0,r)$. The above argument leads instead to 
				\[
				\frac{1}{1-C_0 r}	\frac{dF}{dr} \geq  \int_{  \text{spt}(L) \cap \pi^{-1}(\partial B(p,r))  } ( \frac{1}{|\nabla d_p|  } -1) d\norm{ \langle L, d_p, r\rangle} + r^{-1}  F(r),
				\]
				hence for some constant $C>C_0$, and for $0< r< r_0$ sufficiently small,
				\[
				\begin{split}
						\frac{d}{dr} (e^{Cr} \frac{F(r)}{r})=& \frac{e^{Cr} }{r} ( \frac{dF}{dr} + \frac{C r-1}{r} F   ) \geq \frac{e^{Cr} }{r} ( \frac{dF}{dr} - \frac{1-C_0 r}{r} F   )
						\\
						\geq & \frac{e^{Cr}   (1-C_0r) }{r  }  \int_{  \text{spt}(L) \cap \pi^{-1}(\partial B(p,r))  } ( \frac{1}{|\nabla d_p|  } -1) d\norm{ \langle L, d_p, r\rangle} 
						\\
						\geq  & \frac{1}{r}   \int_{  \text{spt}(L) \cap \pi^{-1}(\partial B(p,r))  } ( \frac{1}{|\nabla d_p|  } -1) d\norm{ \langle L, d_p, r\rangle} ,
				\end{split}
				\]
			which implies the monotonicity formula upon integration.
				\end{proof}

				We now turn to the special Lagrangian case, focusing on the modifications in the arguments.

					\begin{prop}\label{prop:monotonicitySlag}
					(Monotonicity formula) Suppose $ L$ is a special Lagrangian cycle
					of phse $\hat{\theta}= -\frac{\pi}{2}$ inside the almost Calabi-Yau manifold $(X, \omega_\epsilon, \Omega_\epsilon)$, then for any $\epsilon \leq r_1<r_2\leq r_0$, there is some uniform constant $C$ such that
					\[
					\begin{split}
						&	e^{Cr_2}	r_2^{-1} \int_{ \text{spt}(L)\cap \pi^{-1}B_g(p,r_2)} \frac{ d\norm{L}   }{  \sqrt{\det(g)}  }  - e^{Cr_1}	r_1^{-1} \int_{ \text{spt}(L)\cap \pi^{-1}B_g(p,r_1)}\frac{   d\norm{L}     }{  \sqrt{\det(g)}  } 
						\\
						\geq &  
						\int_{    \text{spt}(L) \cap \pi^{-1} (B_g(p, r_2)\setminus B_g(p, r_1)  }  d_p^{-1} ( 1- |\nabla d_p| ) \frac{1}{  \sqrt{\det(g)}  } d\norm{ L}  \geq 0.
					\end{split}
					\]

				\end{prop}

				\begin{proof}
			Special Lagrangians of phase $\hat{\theta}=-\frac{\pi}{2}$ are calibrated by the $n$-form
			\[
			\text{Re}(e^{-i\hat{\theta} } \Omega_\epsilon)= \text{Re}(    \sqrt{-1} \bigwedge_1^n( \epsilon dy_i- \sqrt{-1} dx_i     ) ),
		\]
		with respect to the metric $   (\det(g))^{-1/n}   g_\epsilon  $:
		\begin{equation}\label{eqn:specialLagcalibration}
			\text{Re}(e^{-i\hat{\theta} } \Omega_\epsilon) \leq  \frac{1}{  \sqrt{\det(g) } } \mathcal{H}^n.
		\end{equation}
			We focus on the product ambient geometry case, namely $g_{ij}$ does not depend on $x$. By a linear change of coordinates in $GL(n,\R)$, we can find complex coordinates $x_i' + \sqrt{-1} y_i'$ instead of $x_i+ \sqrt{-1} y_i$, such that $dx_i'$ are orthonormal with respect to $g$. Thus
		\[
		\omega_\epsilon= \epsilon \sum dx_i' \wedge dy_i',\quad  g_\epsilon= \epsilon^2 \sum dy_i'^2+ \sum dx_i'^2,\quad \Omega_\epsilon=\frac{1}{  \sqrt{\det(g) } } \bigwedge_1^n( \epsilon dy_i'- \sqrt{-1} dx_i'     ).
		\]

		We define the vector field $V= \sum_1^n x_i' \frac{\partial}{\partial x_i'}=\sum x_i  \frac{\partial}{\partial x_i}$, and find a primitive $(n-1)$-form for $	\text{Re}(e^{-i\hat{\theta} } \Omega_\epsilon)$:
		\begin{equation}
		\beta= \frac{1}{  \sqrt{\det(g) }  } \int_0^1   \text{Re}(    \sqrt{-1} \iota_V \bigwedge_1^n( \epsilon dy_i'- \sqrt{-1} sdx_i'     ) ) \frac{ds}{s},
		\end{equation}
		using that $d\iota_V ( dx_1\wedge \ldots dx_k)= k dx_1\wedge \ldots dx_k$, and that the integrand is polynomial in $s$ (the coefficient of $s^{-1}$ vanishes), we deduce that $d\beta= 	\text{Re}(e^{-i\hat{\theta} } \Omega_\epsilon)$.

		  When we restricted to $\text{spt}(L)\cap \pi^{-1} (\partial B_g(p,r))$,  the $(n-1)$-form 
\[
\text{Re}(    \sqrt{-1} \iota_V \bigwedge_1^n( \epsilon dy_i'- \sqrt{-1} sdx_i'     ) )  
\]
		is bounded by $sr$ times the $(n-1)$-volume form induced by the metric 
		$\epsilon^2\sum dy_i'^2+ s^2 \sum dx_i'^2$, which is dominated by the metric $g_\epsilon=\epsilon^2\sum dy_i'^2+  \sum dx_i'^2$ when $0\leq s\leq 1$. Hence
		\[
	s^{-1}	\text{Re}(    \sqrt{-1} \iota_V \bigwedge_1^n( \epsilon dy_i'- \sqrt{-1} sdx_i'     ) )  \leq r \mathcal{H}^{n-1}_{g_\epsilon} .
		\]
		Upon integration in $s$, we see that when restricted to $\text{spt}(L)\cap \pi^{-1} (\partial B_g(p,r))$,
		\begin{equation}\label{eqn:primitive2}
			  |\beta |  \leq  \frac{r}{    \sqrt{\det(g) }    }\mathcal{H}^{n-1}_{g_\epsilon} .
			\end{equation}
		This is the analogue of (\ref{eqn:primitive1}).
		
		The rest of the argument is basically the same as the associative case, as long as we replace the measure $d\norm{L}$ by $\frac{1}{ \sqrt{\det(g)}  } d\norm{L}$. This modification ultimately reflects the $\det(g)$ factor in (\ref{eqn:specialLagcalibration}).
				\end{proof}

		\begin{cor}\label{cor:volumebound}
		
We assume the special Lagrangian (resp. associative) cycles $L$ satisfy the uniform linear mass bound in Def. \ref{Def:uniformvolumebound}. Then there is a uniform in $\epsilon$ mass upper bound for $\epsilon\leq r<r_0$,
		\[
		\norm{L}( \pi^{-1}B_g(p, r) ) \leq \begin{cases}
			C\epsilon^{n-1}r ,\quad &\text{special Lagrangian case},
			\\
		C	\epsilon^2 r,\quad  & \text{associative case}.
		\end{cases}
		\]
		Furthermore, if $\pi^{-1}(p)$ contains some point in the support of $L$, then there is a uniform mass lower bound for $\epsilon\leq r<r_0$,
		\[
		\norm{L}( \pi^{-1}B_g(p, r) ) \geq \begin{cases}
			C^{-1}\epsilon^{n-1} r ,\quad &\text{special Lagrangian case},
			\\
			C^{-1}	\epsilon^2r ,\quad  & \text{associative case}.
		\end{cases}
		\]
		\end{cor}

			\begin{proof}
			The volume upper bound follows from the monotonicity formula and Def. \ref{Def:uniformvolumebound}. 
			
			For the lower bound, we observe that the rescaled manifolds $(X, \epsilon^{-2} \omega_\epsilon, \epsilon^{-n}\Omega_\epsilon )$ (resp. $(X, \epsilon^{-3}\phi_\epsilon))$ have uniformly bounded geometry at length scale $O(1)$, so the standard monotonicity formula for calibrated integral currents on geodesic balls implies that
			\[
				\norm{L}( \pi^{-1}B_g(p, \epsilon ) ) \geq \begin{cases}
					C^{-1}\epsilon^n  ,\quad &\text{special Lagrangian case},
					\\
					C^{-1}	\epsilon^3   ,\quad  & \text{associative case}.
					\end{cases}
			\]
			This implies the lower bound for $\epsilon\leq r\leq r_0$ by monotonicity.
			\end{proof}

	\begin{rmk}\label{rmk:logconvergencerate}(\emph{$C^0$-perturbation})
	Suppose we replace the closed $G_2$-structure $\phi_\epsilon$ by $\phi_\epsilon'$, such that 
	$\norm{ \phi_\epsilon - \phi_\epsilon'}_{C^0(X, g_\epsilon)} \leq \delta\ll 1$. Up to rescaling the projection map $\pi: X\to B$ by a factor $(1-C\delta)$, we can still ensure $\pi$ is 1-Lipschitz, hence $|\nabla d_p|\leq 1$.
	Let $F(r)= \norm{L}(\pi^{-1} B_g(p,r))$.  A minor adaption of the proof of Prop. \ref{prop:monotonicityassociative} would give
	\[
	\frac{d F(r)}{dr} \geq  (1-C_0 \delta- C_0 r)   \left(     \int_{  \text{spt}(L) \cap \pi^{-1}(\partial B(t))  } ( \frac{1}{|\nabla d_p| } -1) d\norm{ \langle L, d_p, r\rangle} + r^{-1}  F(r)  \right),
	\]
	hence for some large enough constant $C$, and for $0<r<r_0<1$ sufficiently small,
	\[
	\frac{d}{dr} ( e^{Cr}  \frac{F(r) }{ r^{1-C_0 \delta} }     ) \geq 0.
	\]
	For $\epsilon< r< r_0<1$, the argument in Cor. \ref{cor:volumebound} would give the mass lower bound
	\[
	r^{-1 }    \norm{L}(\pi^{-1} B_g(p,r)) \geq C^{-1}  (\epsilon/r)^{  C_0\delta } \epsilon^2.
	\]
	If we assume that $\delta= O(  \frac{1}{|\log \epsilon|}    )$ as $\epsilon\to 0$, then the factor $ (\epsilon/r)^{  C_0\delta }$ would remain bounded below, so Cor. \ref{cor:volumebound} would still hold.
	\end{rmk}

\section{Limiting $\Lambda$-weighted gradient cycle}\label{sect:limitgradientcycle}

				\subsection{Limiting measure and limiting current}

					Suppose  the special Lagrangian (resp. associative) cycles $L_i$ satisfy the \emph{uniform linear mass bound} as $\epsilon\to 0$ (\cf Def. \ref{Def:uniformvolumebound}). 
				The uniform mass bound in Cor. \ref{cor:volumebound} allows us to extract a subsequential limiting measure with rectifiable support. We shall suppress the indices in diagonal subsequence arguments.

				\begin{cor}(\emph{Limiting measure})\label{cor:limitingmeasure}
			After passing to a subsequence,  the pushforward measures $\epsilon^{1-n}\pi_* \norm{L_i}$ (resp. $\epsilon^{-2} \pi_* \norm{L_i} $) on $B$ converges weakly to some measure $\mu$, such that for $0< r<r_0$, 
					\begin{equation}
						C^{-1} \leq \frac{ \mu(B_g(p,r))}{ r } \leq C,\quad \forall p\in \text{spt}(\mu), B_g(p,r)\subset B.
					\end{equation}
					In particular $\Gamma= \text{spt}(\mu)$ is \emph{rectifiable} by Preiss's theorem.
					
				\end{cor}

			Thus the measure $\mu$ admits the integral representation
			\begin{equation}
				\int f d\mu= \int_\Gamma f (x) \Theta_\mu(x) d\mathcal{H}^1(x), \quad \forall f\in C^0_{cpt}(B).
				\end{equation}
				By the monotonicity formula in Prop. \ref{prop:monotonicitySlag}, \ref{prop:monotonicityassociative}, the density at any $x\in \Gamma$
				\[
				\Theta_\mu(x)= \lim_{r\to 0} \frac{1}{2r} \mu( B(x,r) ) ,
				\]
				is well defined and locally bounded from above and below.

				\begin{cor}\label{cor:Hausdorffdist}
				On any open subset $U$ properly contained in $B$, the support of $\pi_* \norm{L_i}$ converges in the Hausdorff distance to $\Gamma= \text{spt}(\mu)$.
				\end{cor}

				\begin{proof}
				This follows from the definiton of the limiting measure $\mu$, and the mass lower bound in Cor. \ref{cor:volumebound}.
				\end{proof}

				Next, for any given compact set $K\subset B$, we consider the vector space of all smooth $n$-forms (resp. $3$-forms) $\alpha$  which can be written as 
		$
				\alpha=\sum \alpha_i\wedge dx_i
			$
				for smooth $(n-1)$-forms  (resp. $2$-forms) $\alpha_i$ supported on $K$. 	We observe that 
				\[
				|\alpha| \leq  C\sum_i \norm{\alpha_i}_{C^0(g_1)} \epsilon^{1-n }   \mathcal{H}^n_{g_\epsilon} ,\quad \text{resp. } |\alpha | \leq  C \sum_i \norm{\alpha_i}_{C^0(g_1)}\epsilon^{-2 }\mathcal{H}^3_{g_\epsilon}
				\]
			where $g_1$ is a fixed reference metric independent of $\epsilon$, and the constants $C$ are uniform for small $\epsilon$ and for $\alpha$. Thus the sequence of special Lagrangian (resp. associative) cycles $L_i$ satisfy the uniform bounds
			\[
		|	\int_{L_i} \alpha | \leq C\sum_j \norm{\alpha_j}_{C^0(g_1)} \epsilon^{1-n }   \norm{L_i}(K) ,
			\]
				resp.
				\[
					|	\int_{L_i} \alpha | \leq C\sum_j \norm{\alpha_j}_{C^0(g_1)} \epsilon^{-2 }   \norm{L_i}(K).
				\]
				We can view $L_i$ as bounded functionals on the space of all such forms $\alpha$, and take a subsequential weak limit $L_\infty$, so that 
				\begin{equation}\label{eqn:Linftybounded}
					|	\int_{L_\infty} \alpha | \leq C\sum_j \norm{\alpha_j}_{C^0(g_1)} \mu(\text{spt} (\alpha_j)),
				\end{equation}
				where $\mu$ is the limiting measure in Cor. \ref{cor:limitingmeasure}. We can informally think of $L_\infty$ as a current, with the caveat that we only allow test forms $\alpha$ with at least one base factor $dx_i$.

		\begin{lem}\label{lem:limitcurrent1}
		Suppose $\alpha= \sum \alpha_{ij} \wedge dx_i\wedge dx_j$ for some $(n-2)$-form (resp. 1-form) $\alpha_{ij}$ compactly supported in $B$, then $\int_{L_\infty} \alpha=0$.
		\end{lem}

		\begin{proof}
		This follows by taking the limit of the inequality
		\[
			|\alpha| \leq  C\sum_{i,j}\norm{\alpha_{ij}}_{C^0(g_1)} \epsilon^{2-n }   \mathcal{H}^n_{g_\epsilon} ,\quad \text{resp. } |\alpha| \leq  C \sum_{i,j} \norm{\alpha_{ij}}_{C^0(g_1)}\epsilon^{-1 }\mathcal{H}^3_{g_\epsilon},
		\]
		which allows us to gain one more $\epsilon$ factor.
		\end{proof}

		\begin{lem}\label{lem:limitcurrent2}
			Suppose $\alpha= \sum d\beta_i\wedge dx_i$ for some $(n-2)$-form (resp. 1-form) $\beta_i$ compactly supported in $B$, then $\int_{L_\infty} \alpha=0$.
		\end{lem}
		
		\begin{proof}
		This follows by taking the limit of the Stokes formula
		\[
		\int_{L_i} \alpha= \int_{\partial L_i}  \beta_i\wedge dx_i=0.
		\]
		\end{proof}

		Suppose that $\alpha= \sum \alpha_i \wedge dx_i$, where $\alpha_i$ is a smooth $(n-1)$-form (resp. $2$-form) whose restriction to each fibre of $\pi$ is closed. Then $\alpha$ induces a 1-form $[\alpha]$ valued in $H^{n-1}(T^n,\R)$ (resp. $H^2(K3, \R)$). If $\alpha_i$ is fibrewise exact, then we can find a smooth $(n-2)$-form (resp. $1$-form) $\beta_i$, such that $\alpha_i- d\beta_i$ vanishes on each fibre, hence 
		\[
		\alpha_i- d\beta_i= \sum_j \alpha_{ij} \wedge dx_j
		\]
		for suitable forms $\alpha_{ij}$. Then by Lemma \ref{lem:limitcurrent1}, \ref{lem:limitcurrent2},
		\[
		\int_{L_\infty} \alpha= \int_{L_\infty}  \sum dx_i\wedge d\beta_i + \sum dx_i\wedge \alpha_{ij} \wedge dx_j=0.
		\]

	The upshot is that for $\alpha= \sum \alpha_i \wedge dx_i$ with $d\alpha_i=0$ on each fibre, the integral $\int_{L_\infty} \alpha$ depends only on $[\alpha]$, and we denote this as $\langle T, [\alpha]\rangle$. But by (\ref{eqn:Linftybounded}) applied to $\alpha_i= \iota_{\frac{\partial}{\partial x_i}  }\alpha$, we have
	\[
	| \langle T, [\alpha]\rangle   |\leq   C \norm{ \alpha }_{C^0}  \mu(\text{spt}(\alpha)   ),
	\] 
	so $T$ defines an $H_{n-1}(T^n, \R)$ (resp. $H_2(K3,\R)$)-valued 1-current on $B$, with locally finite mass dominated by $C\mu$. By Lemma \ref{lem:limitcurrent2}, its boundary $\partial T=0$. By Cor. \ref{cor:limitingmeasure}, the support $\Gamma= \text{spt}(\mu)$ is 1-rectifiable, and $\mu\leq C \mathcal{H}^1\lfloor_\Gamma$ holds inside any given compact subset of $B$.

	Thus $T$ admits the \emph{integral representation}
			\begin{equation}\label{eqn:integralrepresentationT}
			\langle T, [\alpha]\rangle = \int_\Gamma  \langle [\alpha](x), \delta(x)\otimes  \xi(x)\rangle  d\mathcal{H}^1(x),
		\end{equation}
		where 
		$\xi(x)$ is a $g$-unit vector spanning the approximate tangent space $T_x \Gamma$ 
		for $\mathcal{H}^1$-a.e. $x\in \Gamma$,  and $\delta: \Gamma\to H_{n-1}(T^n,\R)$ (resp. $   \delta: \Gamma\to H_2(K3,\R)    $) is a \emph{locally bounded} function. The $H^{n-1}(T^n,\R)$ (resp. $H^2(K3, \R)$) part of $[\alpha]$ pairs with $\delta(x)$, while the 1-form part of $[\alpha]$ pairs with $\xi(x)$.

		Alternatively, we can pair $T$ with compactly supported 1-forms, and output a vector in $H_{n-1}(T^n,\R)$ (resp. $H_2(K3, \R)$).  To identify $T$ as a $\Lambda$-valued rectifiable current in Def. \ref{Def:gradientcycle}, it remains to show that $\delta$ is valued in the integral lattice $\Lambda$.

		\subsection{Local structure of the limiting current}

		We now analyze the local structure of $T$.
		Let $p\in \Gamma= \text{spt}(\mu)  \subset B$ be any point such that the approximate tangent space $T_p \Gamma$ is unique, and $p$ is a Lebesgue point of $\xi(x), \delta(x), \Theta_\mu(x)$, namely
		\[
		\lim_{r\to 0} r^{-1} \int_\Gamma |\xi(x)- \xi(p)|  +  |\delta(x)- \delta(p)|  + |\Theta_\mu(x)-\Theta_\mu(p)|  d\mathcal{H}^1=0.
		\]
		This holds at $\mathcal{H}^1$-a.e. point $p\in \Gamma$.

		\begin{lem}\label{lem:integrality}
		The class $\delta(p)$ lies in the integral lattice $\Lambda= H_{n-1}(T^n,\Z)$ (resp. $H_2(K3,\Z))$.
		\end{lem}

		\begin{proof}
			We focus on the associative case, as the special Lagrangian case is almost verbatim.
			
		Let $0<\gamma \ll 1$ be any given small number. Let $v= \xi(p)$ denote the unit tangent vector to $ \Gamma $. For all sufficiently small $0<r\ll r_0$, using the uniform mass ratio bounds in Cor. \ref{cor:limitingmeasure}, we can ensure
		\[
		\Gamma\cap B_g(p, 10r) \subset \{   x: \text{dist}_g(x, p+ \R v) < \gamma r        \}.
		\]
		By construction $\mu$ is the weak limit of the pushforward measure $\epsilon^{-2} \norm{L_i}$ as $i\to +\infty$. Using the mass lower bound in Cor. \ref{cor:volumebound}, for sufficiently large $i$ depending on $r$, we have 
		\begin{equation}\label{eqn:suppLiclosetotangentline}
		\text{spt}( \pi_* \norm{L_i}  ) \cap B_g(p, 9r) \subset \{   x: \text{dist}_g(x, p+ \R v) <2 \gamma r        \}.
		\end{equation}
		Let $\pi_{p,v}: B_g(p, 9r)\to p+ \R v\simeq \R$ be the $g$-nearest point projection map; as $r$ is very small, this is well approximated by a Euclidean projection map. We consider the cylindrical subregion
		\[
		C_{p,r}=\{      x\in B_g(p, 9r): \text{dist}_g(  x, p+\R v  ) < r,  \quad \pi_{p,v}(x)=p+tv, \quad |t|< r        \},
		\]
		so that $\pi_{p,v}$ exhibits $C_{p,r}$ as a disc bundle over an interval.

		We restrict the associative cycles $L_i$ to $C_{p,r}$, and take the current slices by the projection $\pi_{p,v}$:
		\[
		L_{i,t}:=  \langle L_i \lfloor  C_{p,r},  \pi_{p,v}, t\rangle ,\quad |t|<r.
		\]
		By (\ref{eqn:suppLiclosetotangentline}), the support of $L_{i,t}$ is contained in 
		\[
		\{      \pi_{p,v}(x)=p+ tv, \quad      \text{dist}_g(x, p+ \R v) <2 \gamma r           \}.
		\]
		Since $\gamma\ll 1$, this ensures $L_{i,t}$ has \emph{no boundary}, so defines a class  $\delta_r \in H_2(K3,\Z)$. This class is independent of $t$ for $|t|<r$.

		On the other hand, by the \emph{slicing theory} of integral currents,
		\[
		(L_i \lfloor  C_{p,r} )  \lfloor  d\pi_{p,v}     = \int_{-r}^{r} 	L_{i,t} dt .
		\]
		Let $\alpha'$ be any given closed 2-form on the K3 surface $\pi^{-1}(p)$, which we regard as a locally defined closed 2-form  on $X$ by identifying the nearby K3 fibres through some diffeomorphism.
		We take a test 3-form $\alpha$, whose restriction to $C_{p,r}$ agrees with $\alpha'\wedge d\pi_{p,r}$. Then 
		\[
		\int_{L_i  \lfloor  C_{p,r}    }   \alpha =  \int_{-r}^{r} 	\int_{L_{i,t} } \alpha' dt = \int_{-r}^{r} \langle [\alpha'], \delta_r\rangle  dt = 2r  \langle [\alpha'], \delta_r\rangle.
		\]
		Passing to the $i\to +\infty$ limit, we obtain
		\[
	\langle  T  \lfloor  C_{p,r}   ,   [ \alpha ]	\rangle =  \int_{ L_\infty  \lfloor  C_{p,r}    }   \alpha  = 2r  \langle [\alpha'], \delta_r\rangle.
		\]
		By the integral representation of $T$ in (\ref{eqn:integralrepresentationT}),
		\[
		\begin{split}
				\langle T \lfloor  C_{p,r} , [\alpha]\rangle = & \int_{\Gamma \cap C_{p,r}  }  \langle [\alpha](x), \delta(x)\otimes  \xi(x)\rangle  d\mathcal{H}^1(x)
				\\
				=&  \int_{  \Gamma \cap C_{p,r}  }  \langle [\alpha'], \delta(x)\rangle  d\pi_{p,v}(\xi(x)) d\mathcal{H}^1.
		\end{split}
		\]
		Comparing the two equations,
		\[
		\delta_r= \frac{1}{2r} \int_{  \Gamma \cap C_{p,r}  }  \delta(x)   d\pi_{p,v}(\xi(x)) d\mathcal{H}^1.
		\]
	Since $p$ is a Lebesgue point of $\delta(x), \xi(x)$, and the approximate tangent space is unique at $p$, we see
	\[
	\delta_r= \delta(p) +o(1),\quad r\to 0.
	\]
		But $\delta_r\in H_2(K3,\Z)$ is an integral class, so the limit class $\delta(p)$ is also integral, and in fact $\delta_r=\delta(p)$ for sufficiently small $r$.
		\end{proof}

		Thus $T$ is a $\Lambda$-valued 1-rectifiable current. We now relate $\delta(p)\in \Lambda$ to the tangent space of $\Gamma$.

		\begin{prop}\label{prop:gradientcycle}
			$T$ is a $\Lambda$-weighted gradient cycle (\cf Def. \ref{Def:gradientcycle}).
		\end{prop}

		\begin{proof}
			We focus on the associative case, and continue with the notations in the proof of Lemma \ref{lem:integrality}. Recall $0< \gamma\ll 1$ is any given small number.  We suppose $0<r<\gamma$, and $r$ is small enough as required in the arguments of Lemma \ref{lem:integrality}. The constants $C$ below are independent of $\gamma, \epsilon, i$.

	%	Moreover, for $i$ large enough depending on $r$, we have $\epsilon <r$, hence by the mass upper bound in Cor. \ref{cor:volumebound}
		%	\[
		%	\text{Mass}(   L_i \lfloor  C_{p,r}     ) \leq  \norm{L_i}( \pi^{-1}(B_g(p, 10r))   )\leq C\epsilon^2 r. 
		%	\]
%Thus by the coarea formula, we can \emph{select a good slice} $r/2< t<r$, such that 
%			\begin{equation}\label{eqn:goodslice}
%				\begin{cases}
%				\text{Mass}   (  {L_{i,t}} ) + 	\text{Mass}   (  {L_{i,-t}} ) \leq C\epsilon^2,
	%			\\
%				\int_{ \text{spt}(L_{i,t}   )  }( 1- |\nabla d_p| ) d\norm{ L_{i,t}}  +    	\int_{ \text{spt}(L_{i,-t}   )  }( 1- |\nabla d_p| ) d\norm{ L_{i,-t}}    \leq C\epsilon^2 \gamma r
%				\end{cases}
%			\end{equation}

	By the coarea formula, and the mass upper bound in Cor. \ref{cor:volumebound}, when $i$ is large enough so that $\epsilon\ll r$,
			\[
			\int_{r/2}^r  \text{Mass}   (  {L_{i,t}} )  dt \leq C \text{Mass}(   L_i \lfloor  C_{p,r}     ) \leq C \norm{L_i}( \pi^{-1}(B_g(p, 9r))   )\leq C\epsilon^2 r. 
			\]
			Similarly
		$
				\int_{r/2}^r  \text{Mass}   (  {L_{i,-t}} )  dt \leq C\epsilon^2 r. 
$
			Thus we can select a  good slice $r/2<t<r$, such that 
			\begin{equation}
			\text{Mass}   (  {L_{i,t}} ) + 	\text{Mass}   (  {L_{i,-t}} ) \leq C\epsilon^2.
			\end{equation}

			Recall that $L_{i,t}\subset   \{   \text{dist}_g(x, p+ \R v) <2 \gamma r   \}   $ as in the proof of Lemma \ref{lem:integrality}. Since $\gamma \ll 1$ is small, $\partial ( L\lfloor C_{p,t} )= L_{i,t}- L_{i,-t}$ as there is no boundary contribution for $|t|<r$. We take local coordinates $x_1, x_2, x_3$ so that $\frac{\partial}{\partial x_i}$ is $g$-orthonormal at $p$, and $x_1$ agrees with $\pi_{p,v}$, so in particular $v= \xi(p)= \frac{\partial}{\partial x_1}$. 
			Then by Stokes formula,
			\begin{equation}
				\begin{split}
					&	\int_{L_i  \lfloor C_{p,t} }(\sum_i \epsilon^2 \omega_i \wedge dx_i - dx_1dx_2dx_3)
					\\
					= & \int_{L_{i,t}} \sum_{cyc}  x_1(\epsilon^2\omega_1- \frac{1}{3}dx_2\wedge dx_3) -   \int_{L_{i,-t}} \sum_{cyc}  x_1(\epsilon^2\omega_1- \frac{1}{3}dx_2\wedge dx_3) .
				\end{split}
			\end{equation}
			Since $\phi_\epsilon$ deviates from the product structure by $O(r)$ relative error, and $r<\gamma$,
			\[
				\int_{L_i   \lfloor C_{p,t}   } \sum_i \epsilon^2\omega_i \wedge dx_i - dx_1dx_2dx_3= \int_{L_i   \lfloor C_{p,t}   } \phi_\epsilon +O(r\norm{L_i}(  C_{i,t}   ) )= \norm{L_i}(  C_{p,t}   ) (1+O(\gamma)).
			\]
			Since $L_{i,t}\subset   \{   x_1=t,  \text{dist}_g(x, p+ \R v) <2 \gamma r   \}   $, we know $|x_2|+ |x_3| \leq C\gamma r \leq C\gamma t$ on the support of $L_{i,t}$, so
			\[
			\begin{split}
					\int_{L_{i,t}} \sum_{cyc}  x_1(\epsilon^2 \omega_1- \frac{1}{3} dx_2\wedge dx_3) =&  t	\int_{L_{i,t}}  (\epsilon^2 \omega_1- \frac{1}{3} dx_2\wedge dx_3)  + O(\gamma  t \text{Mass}   (  {L_{i,t}} )  )
					\\
					=&  t	\int_{L_{i,t}}  (\epsilon^2 \omega_1- \frac{1}{3} dx_2\wedge dx_3)  + O(\gamma  t \epsilon^2 )
					\\
					= & t \epsilon^2 \langle \delta(p) , [\omega_1] \rangle + O(\gamma  t \epsilon^2 )
			\end{split}
			\]
			where the second line uses the mass bound for the good slice, and the third line evaluates the cohomological integral on the closed current $L_{i,t}$, using $[L_{i,t}]= \delta_r=\delta(p)\in H_2(K3,\Z)$ for small enough $r$. Likewise with $L_{i,-t}$.

			Combining these ingredients in the above Stokes formula,
			\begin{equation}\label{eqn:massincylinder}
			\norm{L_i}(  C_{p,t}   )= 2\epsilon^2 t (1+ O(\gamma ) ) \langle \delta(p) , [\omega_1] \rangle + O(\gamma  t \epsilon^2) .
			\end{equation}
		Since $|d\pi_{p,v}|\leq 1+Cr\leq 1+C\gamma$, the slicing inequality gives
				\[
				\int_{-t}^{t}\operatorname{Mass}(L_{i,s})\,ds
				\leq (1+C\gamma)\,\norm{L_i}(C_{p,t}).
				\]
				There is no factor $2t$ in this coarea estimate. Dividing by the length $2t$ of the slicing interval only when taking the average, and then using \eqref{eqn:massincylinder}, we obtain a slice with $|s|<t<r$ for which
				\[
				\begin{split}
					\operatorname{Mass}_{g_\epsilon}(L_{i,s})
					&\leq \frac{1+C\gamma}{2t}\,\norm{L_i}(C_{p,t})\\
					&\leq (1+C\gamma)\epsilon^2\langle\delta(p),[\omega_1]\rangle
					+C\gamma\epsilon^2.
				\end{split}
				\]
			We can use a $(1+C\gamma)$-Lipschitz map to project the closed integral current $L_{i,s}$ into the K3 fibre $\pi^{-1}(p)$ with metric $\epsilon^2 g_{K3}$. This gives a representative of $\delta(p)\in H_2(\pi^{-1}(p),\Z )$, with $g_{K3}$-mass bounded by $(1+ C\gamma)\langle \delta(p) , [\omega_1] \rangle + C\gamma $, hence
			\begin{equation*}
	|\delta(p) |_p \leq (1+ C\gamma)\langle \delta(p) , [\omega_1] \rangle + C\gamma .
			\end{equation*}
			This holds for any given small $\gamma>0$, hence 
$
	|\delta(p)|_p \leq  \langle \delta(p) , [\omega_1] \rangle .
$
Thus the mass minimizer of $\delta(p) \in H_2(\pi^{-1}(p),\Z )$ must be calibrated by the K\"ahler form $\omega_1$, so it is an $I_1$-holomorphic curve (possibly with multiplicity). In particular, for $v= \xi(p)= \frac{\partial}{\partial x_1}$, 
\begin{equation}\label{eqn:densitycalibration}
|\delta(p)|_p = \langle \delta(p) , [\omega_1] \rangle = \langle \delta(p), \nabla_v H\rangle ,
\end{equation}
and $\int_\delta [\omega_2]= \int_\delta [\omega_3]=0$, so that $\langle \delta(p), \nabla H\rangle $ points in the direction of $v$. This verifies the gradient flow condition in Def. \ref{Def:gradientcycle}. We recall that $\partial T=0$, hence $T$ is a $\Lambda$-weighted gradient cycle.

The special Lagrangian case follows the same strategy up to cosmetic changes.
		\end{proof}

\begin{cor}\label{cor:densitymeasure}
The limiting measure $\mu$ agrees with the density measure $\mu_T$ (\cf Def. \ref{Def:densitymeasure}).
\end{cor}

\begin{proof}
We focus on the associative case, and follow the proof of Prop. \ref{prop:gradientcycle}.
By (\ref{eqn:massincylinder})(\ref{eqn:densitycalibration}) and the convergence of $\epsilon^{-2} \norm{L_i}$, we have
\[
\mu(  C_{p,t}  ) = \lim_{i\to +\infty} \epsilon^{-2} \norm{L_i} (C_{p,t}) =     2t (1+ O(\gamma ) ) |\delta(p)|_p + O(\gamma  t )         .
\]
But by asssumption $T_p \Gamma$ is the unique approximate tangent space, and $p$ is a Lebesgue point for $\Theta_\mu(x)$, hence
\[
\Theta_\mu(p)= \lim_{t\to 0} \frac{1}{2t}  \mu(  C_{p,t}  ) =   (1+ O(\gamma ) )|\delta(p)|_p + O(\gamma  ) .
\]
Since this holds for all small $\gamma>0$, we deduce $\Theta_\mu(p)=|\delta(p)|_p$. This holds for $\mathcal{H}^1$-a.e. $p\in \Gamma$, hence $\mu=\mu_T$.

The special Lagrangian case is almost verbatim.
\end{proof}

	Now we prove the main theorem \ref{mainthm}: item 1 follows by combining Cor. \ref{cor:limitingmeasure},  Cor. \ref{cor:densitymeasure}; item 2 follows from Lem. \ref{lem:integrality} and Prop. \ref{prop:gradientcycle}; item 3 follow from item 1 and Cor. \ref{cor:Hausdorffdist}.

		\begin{rmk}
		Our strategy fixes the subsequence $i\to +\infty$ by first extracting the limiting measure $\mu$ and the limiting current $T$. This circumvents taking diagonal subsequences with uncountably many choices corresponding to $p\in B$, and avoids dealing with the simultaneous limits $i\to +\infty$ and $r\to 0$. 
		\end{rmk}

		\section{Regularity of $\Lambda$-weighted gradient cycles}\label{sect:regularity}

		We now study regularity questions for $\Lambda$-weighted gradient cycles, regardless of whether they come from subsequential limits of associative or special Lagrangians.

	\subsection{Characterisation by calibrated condition}

	The \emph{$\Lambda$-weighted gradient cycles} (\cf Def. \ref{Def:gradientcycle}) can be viewed as a \emph{generalisation of calibrated 1-dimensional integral currents}.  In this section, let $T$ be any closed $\Lambda$-valued 1-rectifiable current as in Def. \ref{Def:Lambdavaluedcurrent}, and we write $\Gamma=\text{spt}(T)$.

	We first consider the special Lagrangian case. We have an $H^{n-1}(T^n,\R)$-valued 1-form on the open domain $B\subset \R^n$,
	\[
	\hat{\Omega}=	\sum_i (-1)^{n-i}  dy_1\wedge \ldots dy_{i-1}\wedge dy_{i+1}\wedge \ldots dy_n \wedge  dx_i,
	\]
	which has a natural pairing with any element of $H_{n-1}(T^n,\Z) \otimes_\Z T_p B$. Intuitively  $\hat{\Omega}$ can be viewed as the leading order term in $\text{Re}(\sqrt{-1} \bigwedge_1^n(dy_i- \sqrt{-1} dx_i) )$ that involves the least number of $dx_i$ factors.
	It satisfies the following properties:
	\begin{enumerate}
		\item  The  $H^{n-1}(T^n,\R)$-valued 1-form $\hat{\Omega}$ is $d$-closed.

		\item For $\mathcal{H}^1$-a.e. point $p\in \Gamma$, the weight $\delta(p)\in H_{n-1}(T^n,\Z)$ can be identified as an element of $T_p ^*B$ via contraction with $\hat{\Omega}$. For any unit tangent vector  $v=\sum v^i \frac{\partial}{\partial x_i}$  in $T_p \Gamma$
		with respect to the Hessian metric $g$ on $B$ , we have a calibration inequality  (\cf Example \ref{eg:minareatorus})
		\begin{equation}\label{eqn:pointwisecalibration1}
			\hat{\Omega} ( \delta(p) \otimes v) \leq  \sqrt{g^{ij} \delta(p)_i \delta(p)_j }  \sqrt{ g_{ij} v^i v^j}=  \frac{1}{\sqrt{\det(g)} } |\delta(p)|_p,
		\end{equation}
		with equality iff 
		$ g^{ij} \delta(p)_i \frac{\partial}{\partial x_j}  $
		 points in the direction of $v$, namely $v= v_p(\delta)$.

		The integrated version of this calibration inequality is that
		\begin{equation}\label{eqn:integratedcalibration1}
			\langle  \hat{\Omega},   T \lfloor_U  \rangle \leq  \int_{\Gamma\cap U}  \frac{1}{  \sqrt{\det(g)} }  |\delta(x)|_x d\mathcal{H}^1 (x),
		\end{equation}
		holds for any open subset $U$ properly contained in $B$. The equality is achieved for arbitrary $U$,
		iff $T$ is a $\Lambda$-weighted gradient cycle.

	\end{enumerate}

	Now we turn to the associative case. We have an $H^2(K3,\R)$-valued 1-form induced from the positive section $H: B\to H^2(K3,\R)$, given by 
	\[
	dH= \sum_1^3 \frac{\partial H}{\partial x_i} dx_i =\sum_1^3 [\omega_i] dx_i,
	\]
	which has a natural pairing with any element of $H_2(K3, \Z)\otimes_\Z T_p B$.
	It satisfies the following properties:
	\begin{enumerate}
		\item  The  $H^2(K3,\R)$-valued 1-form $dH$ is $d$-closed.
		
		\item For $\mathcal{H}^1$-a.e. point $p\in \Gamma$, and $v$ the unit tangent vector in $T_p \Gamma$ with respect to the metric $g$ on $B$ , we have 
		\begin{equation}\label{eqn:pointwisecalibration2}
			dH (\delta(p) \otimes v )=\langle \nabla_v H, \delta\rangle     \leq |\delta(p)|_p,
		\end{equation}
		with equality iff $\delta$ is admissible and $v=v_p(\delta)$ (\cf Example \ref{eg:K3holocurve}).

	%	the minimal area representative for $\delta(p)\in H_2(\pi^{-1}(p), \Z)$ is calibrated by the K\"ahler form on the K3 fibre determined by the directional derivative $\nabla_v H$, so that it is a holomorphic 2-cycle. In particular, the gradient vector 
	%	\[
	%	\delta(p)\cdot (\nabla H)(p)= \sum_{i,j} g^{ij} \langle \delta(p), [\omega_i] \rangle \frac{\partial}{\partial x_j}=  |\delta(p) |_p  v
	%	\]
	%	points in the direction of $v$. 

		The integrated version of this calibration inequality is that
		\begin{equation}\label{eqn:integratedcalibration2}
			\langle  dH,   T \lfloor_U  \rangle \leq \mu_T(U)= \int_{\Gamma\cap U}  |\delta(x)|_x d\mathcal{H}^1 (x),
		\end{equation}
		holds for any open subset $U$ properly contained in $B$. The equality is achieved for arbitrary $U$, %then the equality in (\ref{eqn:pointwisecalibration}) holds $\mathcal{H}^1$-a.e. In particular, assuming $\partial T=0$, then the equality case holds 
		iff $T$ is a $\Lambda$-weighted gradient cycle.
	\end{enumerate}

	\begin{cor}
		(\emph{Local minimization})
		Let $U$ be any open subset properly contained in $B$, and let $T$ be a $\Lambda$-weighted gradient cycle. Among all the closed $\Lambda$-weighted 1-rectifiable currents $T'$ on $B$, which coincide with $T$ on the complement of $U$, then $T$ minimizes the functional
		\[
		\int_{  \text{spt}(T)\cap U  } |\delta(x)|_x d\mathcal{H}^1(x)
		\]
		in the associative case, resp.
		\[
			\int_{  \text{spt}(T)\cap U  }  \frac{1}{  \sqrt{\det(g)} }      |\delta(x)|_x d\mathcal{H}^1(x)
		\]
		in the special Lagrangian case.

	\end{cor}
	
	\begin{lem}\label{lem:almoststationary}
		(Almost stationarity)
		Let $T$ be a $\Lambda$-weighted gradient cycle on $B$. For any compactly supported vector field $V$, there is some constant depending on the support, such that
		\[
		|\int_\Gamma  \text{div}_{T_x \Gamma} (V) d\mu_T |\leq C( \text{spt(V)} ) \int_\Gamma |V| d\mu_T.
		\]
	\end{lem}

	\begin{proof}
		We first consider the associative case. Let $V$ be any vector field supported on an open subset $U$ properly contained in $B$. Let $f_t=\exp(tV)$ denote the 1-parameter family of diffeomorphisms generated by $V$, and let $T_t=( f_t)_*(T)$ be the pushforward of $T$. Since the functional on $T_t$ is minimised at $t=0$,
		\[
		\int_{ \Gamma\cap U}  |\delta(x)|_{f_t(x)}  |Df_t (\xi(x))| d\mathcal{H}^1(x) \geq 	\int_{ \Gamma\cap U}  |\delta(x)|_{x} d\mathcal{H}^1(x).
		\]
		Thus
		\[
		\int_{ \Gamma\cap U}  |\delta(x)|_x \text{div}_{T_x \Gamma} (V) d\mathcal{H}^1 +  \int_{\Gamma\cap U}   D_t^-|_{t=0} ( |\delta(x)|_{f_t(x)}  ) d\mathcal{H}^1 \leq 0
		\]
		where $D_t^-$ denotes the left Dini derivative. The min area norm $|\delta|_x$ depends in a Lipschitz way on the variation of the fibrewise metrics, hence 
		\[
		| D_t^-|_{t=0} ( |\delta(x)|_{f_t(x)}  ) |\leq  C |\delta(x) |_x  |V|.
		\]
		Recalling that $\mu_T= |\delta(x)|_x d\mathcal{H}^1\lfloor \Gamma$, and replacing $V$ by $-V$, we deduce
		\[
			|\int_\Gamma  \text{div}_{T_x \Gamma} (V) d\mu_T | \leq C( \text{spt(V)} ) \int_\Gamma |V| d\mu_T.
		\]
		
		In the special Lagrangian case, we replace $\mathcal{H}^1$ by $\frac{1}{  \sqrt{ \det(g)  }  }\mathcal{H}^1$, and similar arguments lead to
		\[
			|\int_\Gamma  \text{div}_{T_x \Gamma} (V)  \frac{1}{  \sqrt{ \det(g)  }  }   d\mu_T | \leq C( \text{spt(V)} ) \int_\Gamma    |V|    \frac{1}{  \sqrt{ \det(g)  }  }  d\mu_T.
		\]
		Replacing $V$ by $V \sqrt{ \det(g)  }$ gives the result.
	\end{proof}

	\begin{cor}\label{Cor:monotonicitygradientcycle}
		(Monotonicity formula)
	Let $T$ be a $\Lambda$-weighted gradient cycle on $B$, and suppose $B_g(p, r_0)$ is contained in a given compact subset of $B$. Then for any $0<r_1<r_2<r_0$, we have
	\[
	\begin{split}
		&	e^{Cr_2}	r_2^{-1} 	\norm{T}(  B_g(p, r_2)  )- e^{C r_1} r_1^{-1} 	\norm{T}(  B_g(p, r_1)  )
		\\
		\geq &  C^{-1}
		\int_{    \text{spt}(T) \cap (B(p,r_2)\setminus B(p,r_1) ) }  d_p^{-1} |\nabla^\perp d_p|^2  d\norm{ T}  \geq 0.
	\end{split}
	\]
	In particular, 
\[
\Theta(p):= \lim_{r\to 0} \frac{1}{2r}  \mu_T( B_g(p,r)  )
\]
is well defined at every $p\in \text{spt}(T)$.
	\end{cor}

	\begin{cor}\label{cor:noncollapsing}
		(Local noncollapsing) For any $B(p,r)$ contained in a given compact subset of $B$, there is a uniform lower bound
		\[
		\mu_T(  B_g(p,r) ) \geq C^{-1} r.
		\]
	\end{cor}

	\begin{proof}
		For $\mathcal{H}^1$-a.e. $p\in \Gamma$, the density $\Theta(p)= |\delta(p)|_p$ has a uniform lower bound on any compact subset of $B$. The monotonicity formula implies
		\[
		\frac{1}{r}e^{Cr}	\mu_T(  B_g(p,r) ) \geq \Theta(p)\geq C^{-1},
		\]
		so there is a lower bound $\mu_T(   B_g(p,r)   )\geq C^{-1} r$. This extends to all the points $p\in \Gamma$ with $B_g(p,r)$ contained in the compact set, by a limiting argument.
	\end{proof}

		\subsection{Tangent cone}

		From henceforth	let $T$ be a $\Lambda$-weighted gradient cycle in $B$, and denote $\Gamma= \text{spt}(T)\subset B$. We shall use the tangent cones of $T$, to study the local graph-like structure of $T$.
	Let $p\in \Gamma $, and let \[
	\lambda_{p,r} : B\to T_pB,    \quad \lambda_{p,r}(x)= \frac{x-p}{r}.
	\]
	  By the monotonicity formula, for small enough $r>0$, the sequence of closed $\Lambda$-valued 1-currents $(\lambda_{p,r})_* T$ have uniformly bounded mass on fixed balls in $T_p B$. By White's generalisation of the Federer-Fleming compactness theorem for integral currents with value in discrete normed groups \cite{White}, we can take subsequential limits as $r\to 0$, to obtain some closed $\Lambda$-valued 1-rectifiable current $ C_p$ inside $T_p B$, called a \emph{tangent cone} of $T$ at $p$.

		\begin{prop}\label{prop:tangentcone}
			(Structure of tangent cone)
			The tangent cone $C_p$ is given by a finite $\Lambda$-weighted sum of the current of integration along distinct rays:
		\begin{equation}\label{eqn:tangentcone}
			C_p= \sum_i  \delta_i\otimes   [\R_+ v_p(\delta_i)] ,\quad \delta_i\in \Lambda,
		\end{equation}
		subject to the \emph{balancing condition}:
		\begin{equation}\label{eqn:tangentcone2}
			\sum_i \delta_i=0\in \Lambda.
		\end{equation}
		Furthermore, the total mass of $C_p$ on the unit ball is 
		\begin{equation}\label{eqn:densityboundondelta}
			\Theta(p)= \frac{1}{2}  \sum_i |\delta_i|_p.
		\end{equation}
		Each $\delta_i$ appearing in the tangent cone $C_p$  satisfies $|\delta_i|\leq \Theta(p)$.
	Moreover in the associative case, all $\delta_i$ are admissible (\cf Example \ref{eg:minareatorus}, \ref{eg:K3holocurve}.)
		\end{prop}

		\begin{proof}
			We focus on the associative case. The $\Lambda$-weighted gradient cycle $T$ is characterised among closed $\Lambda$-valued rectifiable currents, by achieving the equality for the calibration inequality (\ref{eqn:integratedcalibration2}). Using the semicontinuity of mass, we can pass this property to the weak limit, to see that $C_p$ achieves the equality for a version of calibration inequality: for any open subset $U$ properly contained in $B$,
			\[
			\int_{ \text{spt}(C_p) \cap U}  |\delta(x)|_x d\mathcal{H}^1(x)= \int_{  \text{spt}(C_p) \cap U }  \langle dH(p), \delta(x)\otimes \xi(x)\rangle d\mathcal{H}^1(x).
			\]
Thus at $\mathcal{H}^1$-a.e. $x\in \text{spt}(C_p)$, 
			\[
			|\delta(x)|_x = \langle dH(p), \delta(x)\otimes \xi(x)\rangle ,
			\]
			hence $\delta(x)$ is \emph{admissible}, and $\xi(x)= v_p(\delta(x))$ (\cf Example \ref{eg:K3holocurve}).

			The monotonicity formula for $T$ implies that $\text{spt}(C_p)$ is conical, hence it is a \emph{union of rays} in $T_p B$ based at the origin. Since $\partial C_p=0$, the $\delta(x)\in \Lambda$ is constant along each ray. Thus the unit tangent vector  $v_p(\delta)$ coincides with the direction of the ray up to sign. 
		Using the flexibility to simultaneously reverse the sign of $\delta(x)$ and $\xi(x)$ without changing $T$, we can ensure the ray agrees with $\R_{\geq 0} v_p(\delta)$.
			This shows the representation formula 
			\[
				C_p= \sum_i  \delta_i\otimes   [\R_+ v_p(\delta_i)] ,\quad \delta_i\in \Lambda.
			\]
			Each unit ray has length one, hence the mass inside the unit ball is $\sum_i|\delta_i|_p$. By the tangent-cone construction it equals
				$\lim_{r\downarrow0}r^{-1}\mu_T(B_g(p,r))=2\Theta(p)$, which proves \eqref{eqn:densityboundondelta}.
			(We remark that when the tangent cone $C_p$ is supported on a line $\R v_p(\delta(p))$, then there are two ray directions, so $\Theta(p)= \frac{1}{2}\sum_i |\delta_i|_p= |\delta(p)|_p$.)

			 Since  all $|\delta_i|_p$ are bounded positively from below, the sum in (\ref{eqn:tangentcone}) must be finite. The balancing condition (\ref{eqn:tangentcone2}) follows from 
			 $\partial C_p=0$ at the origin.

	By the balancing condition (\ref{eqn:tangentcone2}),
			 \[
			 |\delta_i|_p =|\sum_{j\neq i} \delta_j |_p \leq \sum_{j\neq i} |\delta_j|_p,
			 \]
			hence $|\delta_i|_p\leq \frac{1}{2} \sum_j |\delta_j|_p  = \Theta(p)$.

			The special Lagrangian case is almost verbatim.
		\end{proof}

\begin{cor}
At any given $p\in \Gamma$, there is a \emph{unique tangent cone} to $T$.
\end{cor}

\begin{proof}
By (\ref{eqn:densityboundondelta}) there are only finitely many choices of $\delta\in \Lambda$, so there are only finitely many candidates for the tangent cones $C_p$ of the form (\ref{eqn:tangentcone}). But the space of tangent cones at $p$ is connected in the flat topology, while the tangent cone $C_p$ cannot deform, which implies its uniqueness.
\end{proof}

	In particular, at \emph{every} point on $\Gamma= \text{spt}(T)$,  we can refer to $\delta(x) $ as any member of the finitely many $\delta_i$ that appear in the tangent cone.  (At $\mathcal{H}^1$-a.e point on $\Gamma$, this $\delta(x)\in \Lambda$ is unique up to sign.)

		\subsection{Local structure around the tangent cone}

			We now study the local structure of a $\Lambda$-weighted gradient cycle $T$ close to a given point $p\in B$. We focus on the associative case, as it contains all the difficulties present in the special Lagrangian case, and a number of additional subtleties.

			We let $A$ denote the finite set of admissible $\delta\in H_2(K3, \Z)$ with $|\delta|_p \leq \Theta(p)$. This constrains the possible weights $\delta(x)$ for nearby $x\in B$.
			
				\begin{rmk}\label{rmk:reversedelta}
				We note that if $\delta$ is admissible, then $-\delta$ is also admissible, since an $I_v$-holomorphic curve with orientation reversed, is an $-I_v$-holomorphic curve. 	
				Hence $A$ is symmetric under $\delta\to -\delta$.
			%	If $\delta $ is admissible on the fibre $\pi^{-1}(p)$, it may not be admissible on the nearby fibres, because holomorphic curves do not necessarily deform to nearby fibres. 
			\end{rmk}
			
			\begin{lem}
   There is some small neighbourhood $U$ of $p\in B$, such that whenever $x\in U$,  the weight $\delta(x)$ for $T$ lies inside $A$.
			\end{lem}
			
			\begin{proof}
	By Prop. \ref{prop:tangentcone} we have a uniform bound for $|\delta(x)|_x\leq \Theta(x) \leq C$ in a given neighbourhood of $p$, so there are only finitely many possibilities for $\delta(x)\in \Lambda$. Suppose $\delta=\delta(x_i) $ for some sequence $x_i\to p$. Then by the Lipschitz continuity of $|\delta|_x$ in $x$, and the semi-upper continuity of density,
		\[
		|\delta|_p = \lim_{i\to +\infty} |\delta(x_i)|_{x_i}  \leq  \lim_{i\to +\infty} \Theta(x_i) = \Theta(p).
		\]
		Furthermore, the admissibility of $\delta$ on the K3 fibre $\pi^{-1}(x_i)$ implies
		\[
		|\delta|_{x_i}= \langle (\nabla_{v_i} H)(x_i), \delta\rangle 
		\]
		for $g$-unit vectors $v_i\in T_{x_i} B$. Passing to subsequential limits, we obtain a $g$-unit vector $v\in T_p B$, such that 
	$
			|\delta|_p = \langle (\nabla_{v} H)(p), \delta\rangle ,
	$
		so $\delta$ is admissible at $p$. Thus $\delta\in A$. By shrinking $U$, we can thus guarantee $\delta(x)\in A$ for all $x\in U$.
			\end{proof}

		Using the geodesic coordinates, we can identify $U$ with a small open neighbourhood of $0\in T_p B\simeq \R^3$.  The metric $g$ is $C^\infty$-approximated by the Euclidean metric. The set of directions
	$
		\{      v_p(\delta):  \delta\in A     \}
$
is a finite subset of $S^2$, and the possible velocity vectors
\[
\xi(x)= v_x(\delta)=  \frac{  \delta\cdot (\nabla H)(x)   }{     |\delta|_x     }
\]
satisfy $|v_x(\delta)- v_p(\delta)| \leq C |x|$. Let $w\in T_p B$ be a unit vector. As $A$ is symmetric under $\delta\to -\delta$ (See Remark \ref{rmk:reversedelta}), 
 we can write $A=A_w^+\sqcup -A_w^+\sqcup A_w^0$, 
 \[
 \begin{cases}
 		A_w^+= \{  \delta\in A:    g(v_p(\delta), w)>0      \},
 		\\
 		A_w^0=   \{  \delta\in A:    g(v_p(\delta), w)=0      \}.
 \end{cases}
 \]
There is some $\theta_0>0$, such that 
\[
\angle (w, v_p(\delta) ) < \frac{\pi}{2}- \theta_0  ,\quad    \forall \delta\in A_w^+,
\]		
For $\frac{1}{3}\theta_0   <\theta<\frac{2}{3}\theta_0$, we take the cone
\[
C_{q, w, \theta}= \{    x\in \R^3\setminus \{  0\} :   \angle( w, x-q )< \frac{\pi}{2}- \theta          \}.
\]

\begin{lem}\label{lem:cone}
The following holds for sufficiently small $r>0$. Suppose $q\in \Gamma\in B(0, r)$, and the tangent cone $C_q$ of $T$ contains some ray with weight $\delta\in A_w^+$. Then for any $0<t< r$, there is at least one point on $\Gamma\cap C_{q,w,\theta_0/3}\cap \partial B(q,t)$.

\end{lem}

\begin{proof}
	Without loss $B(0,10r)\subset U$, and recall  $\delta(x)\in A$ for $x\in U$. 
By taking $r$ sufficiently small, we can ensure that on $x\in \Gamma\cap B(0,10r)$, the drift in direction
$
| v_x(\delta)- v_p(\delta) |
$
is small enough so that
\[
\begin{cases}
	\angle (w, v_x(\delta) ) < \frac{\pi}{2}-  \theta_0+ Cr <  \frac{\pi}{2}- \frac{99}{100}\theta_0   ,\quad   & \forall \delta\in A_w^+,
	\\
	|\angle (w, v_x(\delta)) - \frac{\pi}{2} | <   Cr  <   \frac{1}{100}\theta_0        ,\quad   &  \forall \delta\in A_w^0.
\end{cases}
\]		
For any $\frac{1}{3}\theta_0   <\theta<\frac{2}{3}\theta_0$,
at any point $x\in \Gamma \cap B(0,10r)\cap \partial C_{q,w,\theta}\setminus \{ q\}$, the tangent vector $v_x(\delta)$ is therefore quantitatively \emph{transverse} to $\partial C_{q,w,\theta}$. More precisely, let $d_q$ be the distance function to $q$, then $|\nabla^\perp d_q|(x) \geq \sin ( \frac{\theta_0 }{10}  )\geq C^{-1}\theta_0$ at such an intersection point.

 But for $r$ sufficiently small,  the volume ratio is bounded on $B(q, 9r)$, so 
the monotonicity formula in Cor. \ref{Cor:monotonicitygradientcycle} implies
\[
\begin{split}
	\int_{    \Gamma \cap B(q, 9r) ) }  d_q^{-1} |\nabla^\perp d_q|^2  d\norm{ T} \leq  C.
\end{split}
\]
Since there is a positive lower bound on $\Theta(x)$ (due to the lower bound on $|\delta|_x$),
\[
	\int_{    \Gamma \cap B(q, 9r) ) }  d_q^{-1} |\nabla^\perp d_q|^2 (x) d\mathcal{H}^1(x) \leq  C.
\]
By the lower bound $|\nabla^\perp d_q|(x) \geq C^{-1}\theta_0$ on $\partial C_{q,w,\theta} \cap \Gamma \cap B(0,10r)$, we deduce
\[
\int_{    \Gamma \cap B(q, 9r) )  \cap (C_{q,w,\theta_0/3}\setminus C_{q,w,2\theta_0/3} ) }  d_q^{-1}  |\nabla^\perp d_q|   d\mathcal{H}^1(x) \leq  C\theta_0^{-1}.
\]
By the coarea formula, we can select some angle parameter $\theta\in ( \frac{1}{3}\theta_0, \frac{2}{3}\theta_0   )$, 
such that the count of (automatically transverse) intersection points
\begin{equation}\label{eqn:countingestimate}
\mathcal{H}^0(       \Gamma \cap B(q, 9r)  \cap \partial C_{q,w,\theta}     ) \leq C\theta_0^{-2}.
\end{equation}
Furthermore, for Lebesgue a.e. choice of $\theta$, we can ensure $T_x\Gamma$ is a copy of $\R$.

We now suppose for contradiction that there is some $0<t< r$, such that 
\[
\Gamma\cap C_{q,w,\theta_0/3}\cap \partial B(q,t)= \emptyset.
\]
We consider the current 
$
T_t= T\lfloor (C_{q,w,\theta}\cap B(q,t)  ),
$
whose boundary $\partial T_t$ receives contribution only from $\partial C_{q,w,\theta}\cap B(q,t)$.

Since we can simultaneuously reverse the sign of $\delta(x)$ and $\xi(x)$ without changing $T$, we can assume that at every $x\in \Gamma \cap B(0,10r)\cap \partial C_{q,w,\theta}\setminus \{q\}$, the vector $\xi(x)$ \emph{points inside}  $C_{q,w,\theta}$. 
Likewise, recall that at $q$, the orientation of the rays determines the sign of $\delta_i(q)\in A$. The key observation is that $v_x(\delta)$ \emph{points outside} $C_{w,q,\theta}$ for $\delta\in -A_w^+$, so that $\xi(x)=v_x(\delta)\in A_w^+\cup A_w^0$.
Thus
\[
\begin{split}
-\partial T_t=& \sum_{ x\in \Gamma \cap B(q,t)\cap \partial C_{q,w,\theta}\setminus \{q\} , \delta(x)\in A_w^0, \text{inward pointing}}  \delta(x) [x]
\\
&
+ \sum_{ x\in \Gamma \cap B(q,t)\cap \partial C_{q,w,\theta}\setminus \{q\} , \delta(x)\in A_w^+}  \delta(x)[x]
+ \sum_{ \delta_i(q)\in A_w^+}  \delta_i(x) [q].
\end{split}
\]
Since $\partial T_t$ is exact, we get an identity in $\Lambda$ (`Kirchhoff's Law'),
\[
\begin{split}
&   - \sum_{ x\in \Gamma \cap B(q,t)\cap \partial C_{q,w,\theta}\setminus \{q\} , \delta(x)\in A_w^0, \text{inward pointing}}  \delta(x)
	\\
	=&  
	\sum_{ x\in \Gamma \cap B(q,t)\cap \partial C_{q,w,\theta}\setminus \{q\} , \delta(x)\in A_w^+}  \delta(x)
	+ \sum_{ \delta_i(q)\in A_w^+}  \delta_i(x) .
\end{split}
\]
Note that while some of these sums may be empty, the term $\sum_{ \delta_i(q)\in A_w^+}  \delta_i(x) $ is nonzero by assumption.
We take the dot product with $\nabla H(p)$, and take the inner product with $w$. Each term from the $A_w^0$ collection is bounded by 
\[
C | g(v_x(\delta), w) |=  |C g( v_x(\delta)- v_p(\delta),   w ) |\leq C r ,
\]
while each nonzero contribution from the $A_w^+$ collection is bounded below by $C^{-1} \theta_0$. Thus 
\[
Cr \# \{ x\in     \Gamma \cap  B(q, t) \cap \partial C_{q,w,\theta}:  \delta\in A_w^0          \} \geq C^{-1}\theta_0.
\]
Contrasting this with the counting estimate (\ref{eqn:countingestimate}), we see
\[
\theta_0^3 \leq Cr,
\]
which is a contradiction for small enough $r$. This proves  the lemma.
\end{proof}

	We now partition $A$ into a disjoint union $\sqcup A_k$, where each $A_k$ collects together all $\delta\in A$ associated to the same $v_p(\delta)\in T_p B$. Geometrically, this means all the holomorphic curve classes for the \emph{same complex structure} $I_v$, with mass bounded by $\Theta(p)$. 	We recall that the tangent one $C_p$ is a weighted sum of distinct rays:
	\[
	C_p= \sum_i  \delta_i\otimes   [\R_+ v_p(\delta_i)] .
	\]
	This suggests that we should be able to locally decompose $T$ near $p$ into a finite sum $T=\sum T_i$, with each piece $T_i$ modelled on $ \delta_i\otimes   [\R_+ v_p(\delta_i)] $. As above, we identify small balls around $p$ as open neighbourhoods of $0\in T_p B=\R^3$.

	\begin{thm}\label{thm:localstructure}
		(Local structure for $T$) 
	Let $T$ be a $\Lambda$-weighted gradient cycle on $B$, and $p\in \text{spt}(T)$, then
	the following holds on a small enough neighbourhood $U$ around $p$. We can decompose $T$ as a sum $T= \sum_i T_i$, where each $T_i$ is a $\Lambda$-valued 1-rectifiable current, which is a $\Lambda$-weighted gradient cycle on $U\setminus \{ p\}$, and satisfy the following properties:
		\begin{enumerate}
			\item  The boundary of $T_i$ is $\partial T_i= - \delta_i \otimes [p]$.
			
			\item  The tangent cone of $T_i$ at $p$ is $  \delta_i\otimes   [\R_+ v_p(\delta_i)]  $.

			\item   At every $x\in \text{spt}(T_i)$, up to suitable sign choice for $\delta(x)$, we have
			$v_p(\delta(x))= v_p(\delta_i)$. (If  the tangent cone at $x$ has more than one direction, then $\delta(x)$ refers to the $\Lambda$-weight associated to any ray.)

				\item (Second order tangency) If we choose a coordinate system $x_1, x_2, x_3$ with $\frac{\partial}{\partial x_i}$ orthonormal at $p$, and $v_p(\delta_i)= \frac{\partial}{\partial x_i}$ at $p$, then on a sufficiently small neighbourhood of $p$,
			\[
			\text{spt}(T_i)\subset \{     \sqrt{x_2^2+x_3^2} < C x_1       \}.
			\]

			\item  For each value of $x_1>0$, the vertical slice of $T_i$ consists of finitely many points, such that $\sum_x \delta(x)=\delta_i \in \Lambda$, and
				\[
			\sum_x  |\delta(x)|_p = (\nabla_{ v_p(\delta_i)} H)(p)\cdot \delta_i= |\delta_i |_p. 
			\]
			(Here if the tangent cone $C_x$ at $x$ has more than one direction, then we sum over the  $\Lambda$-weights $\delta$ for all the rays of $C_x$ whose outward pointing direction $v_x(\delta)$ is compatible with the orientation choice $v_p(\delta)=v_p(\delta_i)$.)

		\end{enumerate}
	\end{thm}

			\begin{proof}
				We restrict attention to sufficiently small balls $B(0, 10r)$, and let $\gamma>0$ be a small number, so that the cones
				\[
				R_i= \{    x\in   \R^3\cap B(0,10r)  : \angle (x ,  v_p(\delta) ) < \gamma    \}
				\]
				are all disjoint. Since $C_p$ is the unique tangent cone at $p$, using the local noncollapsing estimate in Cor. \ref{cor:noncollapsing}, the support of $T$ is contained in the union of $R_i$. We let $T_i = T\lfloor R_i$. Clearly $T_i$ is a $\Lambda$-weighted gradient cycle on $B(p,r)\setminus \{ p\}$.

			The tangent cone of $T$ is the sum of the independent contributions from the tangent cones of $T_i$, hence item 2 holds. This also implies item 1.

			For item 3, we suppose there is some $q \in \text{spt}(T_i)\cap B(0, r)$ such that $v_p(\delta(q) )$ does not agree with $\pm v_p(\delta_i)$, where $\delta(q)$ refers to  the $\Lambda$-weights associated to some ray in the tangent cone $C_q$. We can then find some $w\in \R^3$,  so that $\delta_i\in A_w^0$, but $\delta(q)\in A_w^+$. Assuming that $r$ is sufficiently small, then by Lemma \ref{lem:cone}, 
			 for any $0<t< r$, there is at least one point on $\text{spt}(T)\cap C_{q,w,\theta_0/3}\cap \partial B(q,t)$. By choosing sufficiently small $\gamma>0$ depending on $\theta_0$, this would contradict the fact that $T$ is supported on $\cup_i R_i$.

			Thus $v_p(\delta(q))$ must agree with $\pm v_p(\delta_i)$. Using the flexibility to reverse the signs of both $\xi(x)$ and $\delta(x)$ simultaneously without changing $T$, we can arrange that $v_p(\delta(q)) = v_p(\delta_i)$. This proves item 3.

			Since $|v_x(\delta)- v_p(\delta)| \leq C|x|$, the component of the velocity $\xi(x)= v_x(\delta(x))$ perpendicular to $v_p (\delta(x)) = v_p (\delta_i)$ is $O(|x|)$. This implies item 4.

			For item 5, we first note that on the vertical slices for a.e. $x_1>0$, the tangent cone $C_x$ is supported on a line, and there is no multi-valued issue for $\delta(x)$. Since for any $0<s_1<s_2$, the vertical slices for $x_1=s_1$ and $x_1=s_2$ differ by the boundary of $T_i \lfloor \{  s_1< x_1<s_2     \}$, we see $\sum_x \delta(x)\in \Lambda$ is constant for all these generic $x_1$. This constant must be $\delta_i$ by item 2.

		Since $v_p(\delta(x))=v_p(\delta_i)$, we have
$
			\delta(x)\cdot (\nabla_{v_p(\delta_i)} H)(p) = |\delta(x) |_p,
		$
			so on these generic vertical slices,
			\[
			\sum_x  |\delta(x)|_p = \sum_x 	\delta(x)\cdot (\nabla_{v_p(\delta_i)} H)(p)= \delta_i\cdot  (\nabla_{v_p(\delta_i) } H)(p)= |\delta_i|_p.
			\]
			In  particular, the mass for the slices are uniformly bounded.

			To extend the result to all $x_1$, we use the fact that non-generic slices arise as  the limits of the generic slices.
			\end{proof}

	\begin{proof}[Proof of Theorem \ref{thm:regularitySlag}]
		Fix $p\in\Gamma$. By Theorem \ref{thm:localstructure}, we can write $T=\sum_iT_i$ near $p$, and the coefficients on $T_i$ satisfy $v_p(\delta(x))=v_p(\delta_i)$. In the special Lagrangian case, Example \ref{eg:minareatorus} shows that these coefficients lie on the same positive rational ray. Hence $v_x(\delta(x))=v_x(\delta_i)$, and the support of $T_i$ is an integral curve of this smooth nonvanishing vector field. The local structure theorem gives only finitely many such sectors, with no accumulating branches. Thus a neighbourhood of $p$ is a finite embedded graph. The result follows.
	\end{proof}

We now prove Theorem \ref{thm:regularityassociative}. %(The special Lagrangian case Thm \ref{thm:regularitySlag} can be proved in an almost identical way; the main simplification  is that $v_p(\delta)= v_p(\delta')$ iff $\delta$ is a positive multiple of $\delta'$.)

\begin{proof}
(Thm. \ref{thm:regularityassociative}) We recall the following technical assumption: whenever there is some choice of complex structure $I$ on the hyperk\"ahler K3 surface $\pi^{-1}(p)$, such that two classes $\delta, \delta'\in H_2(K3,\Z)$ both admit an $I$-holomorphic curve representative, with areas $|\delta|_p, |\delta'|_p \leq \Theta(p)$, then $\delta, \delta'$ are proportional.

By Theorem \ref{thm:localstructure},
in a sufficiently small neighbourhood of $p$, all these $\delta(x)\in A\subset H_2(K3, \Z)$, so admit some mass minimizer on $\pi^{-1}(p)$ which is holomorphic with respect to the complex structure determined by $v_p(\delta(x))$, and has mass bounded by $\Theta(p)$. Moreover, for any $x\in \text{spt}(T_i)$, the $\Lambda$-weight $\delta(x)$ satisfies $v_p(\delta(x))= v_p(\delta_i)$, so all these holomorphic curves share the same choice of complex structure. By the technical assumption, these $\delta(x)$ must be a $\Q$-multiple of $\delta_i$. But this implies that $v_x(\delta(x))= v_x(\delta_i)$ at all $x\in \text{spt}(T_i)$, which is proportional to $(\nabla H)(x)\cdot \delta_i$. Thus the support of $T_i$ is an integral curve of the vector field $\nabla (H\cdot \delta_i)$, namely a gradient flowline of $H(x)\cdot \delta_i$.

This shows that a sufficiently small neighbourhood of $p$ is an embedded graph with one vertex at $p$, and a finite number of edges which are gradient flowlines associated to the weights $\delta_i$.
\end{proof}

\begin{rmk}
Suppose $H: B\to H^2(K3,\R)$ is a generic choice of smooth positive section. If $\delta, \delta'$ are two $\Q$-linearly independent classes, then the locus where the two vector fields $\nabla (H\cdot \delta)$ and $\nabla (H\cdot \delta')$ are proportional, should be a (possibly empty) codimension two subset in $B$. It is a separate question which classes admit holomorphic curve representatives. Under a uniform bound on $\Theta(p)$, only finitely many classes $\delta$ are involved. The technical assumption would hold if $\text{spt}(T)$ avoids this codimension two locus, and then Thm. \ref{thm:regularityassociative} would conclude that $T$ is locally a balanced gradient graph, with finitely many vertices within any compact region.
\end{rmk}

		\subsection{Open questions}

		We mention some open questions.

		\begin{Question}
		Can we drop the technical assumption  in Thm. \ref{thm:regularityassociative} about the $H_2(K3,\Z)$ classes?
		\end{Question}

		For comparison, Allard-Almgren \cite{AllardAlmgren} developed a celebrated regularity theory for \emph{1-dimensional stationary varifolds} in Riemannian manifolds, \emph{not necessarily assuming the integrality of the density $\Theta(x)\in \R_+$}. Under the assumption that $\Theta(x)$ only takes \emph{discrete} values, they prove that the support of the varifold is  an \emph{embedded graph} whose edges are geodesics, and the number of vertices is locally finite. When the discreteness assumption is dropped, they produced an example of a point with \emph{infinite complexity}.

		In our setting, the $\Lambda$-weighted gradient cycle shares the feature that the generalised mean curvature is locally $L^\infty$ (\cf Lemma \ref{lem:almoststationary}), and the density $\Theta(x)$ in general only takes value in real numbers. The lattice $\Lambda$ is discrete, while the density $\Theta(x)\in \R$ in general varies continuously with $x$. When the technical assumption is dropped, the conceivable bad behaviour is that several gradient flowlines with distinct $\delta(x)\in \Lambda$ may be approximately parallel, with a possibly infinite number of intersection points that may accumulate somewhere.

		The specific advantage of our setting is that the gradient flowlines obey first order equations, instead of second order, and at any given point, the number of possible tangent directions is finite (due to the finiteness of holomorphic curve classes $\delta$ in the K3 surface subject to bounded area). This raises the hope that infinite complexity examples may not occur, or can be ruled out under real analyticity or genericity assumptions on the positive section $H$.

		\begin{Question}
			What can we say about the limiting current $L_\infty$, beyond its homological information captured by the $\Lambda$-weighted gradient cycle $T$?
		\end{Question}

		For instance, one can hope that $L_\infty$ is in fact an integral current, and the fibrewise slices of $L_\infty$ along the gradient flowlines are $I_v$-holomorphic curves with homology class $\delta$ contained in the K3 fibres in the associative case (\cf Example \ref{eg:K3holocurve}), resp. the sum of translated copies of flat subtori in the class $\delta$ contained in the $T^n$-fibres in the special Lagrangian case (\cf Example \ref{eg:minareatorus}). Furthermore, one would like to know how these fibrewise slices vary within the moduli space of holomorphic curves.

The main problem is that given the sequence of special Lagrangian (resp. associative) cycles $L_i$, one needs to control how fast the slices of $L_i$ can oscillate in the fibre direction, over a macroscopic base length scale of order one.
If this oscillation is too wild, then it is not even clear that the slices of $L_\infty$ are integral cycles in the fibres, instead of just a closed normal current, such as an $\R$-weighted average over a moduli space of holomorphic curves. In contrast, the $\Lambda$-valued current $T$ is better behaved, because homological classes are prevented from wild oscillation.

	%	real analyticity of the base? issue: tangency for the integral curves associated to two different fibre classes

%caveat: drift in the moduli of fibrewise holo curve

		\begin{Acknowledgement}
			The author is supported by the Royal Society URF. 	Question \ref{Question} seems to be quite well known within the special holonomy community, and the author thanks Y-S Lin ,  S. Esfahani  and Y. Zhang for bringing this up in discussions, and Y. Zhang for comments. 
		\end{Acknowledgement}

	\end{document}